\documentclass[11pt]{amsart}
\usepackage[utf8]{inputenc} 
\usepackage[pctex32]{graphics}
\usepackage{amsfonts}
\usepackage{amssymb,amscd,latexsym}
\usepackage{amsmath}
\usepackage{epsfig}
\usepackage{color}

\newtheorem{thm}{Theorem}[section]
\newtheorem{cor}[thm]{Corollary}
\newtheorem{prop}[thm]{Proposition}

\newtheorem{defin}[thm]{Definition}

\newtheorem{lema}[thm]{Lemma}
\newtheorem{rmk}[thm]{Remark}

\newtheorem{ex}[thm]{Example}

\def\N{\mathbb{N}}

\def\P{\mathbb{P}}
\def\K{\mathbb{K}}

\begin{document}

\title{Counting Square free Cremona monomial maps}
\author[B. Costa]{Barbara Costa}
\author[T. Dias]{Thiago Dias}
\author[R. Gondim]{Rodrigo Gondim*}
\date{}

\begin{abstract}
We use combinatorics tools to reobtain the classification of monomial quadratic Cremona transformations in any number of variables given in \cite{SV2} 
and to classify and count square free cubic Cremona maps with at most six variables, up to isomorphism. 

\end{abstract}

\thanks{*Partially supported  by the CAPES postdoctoral fellowship, Proc. BEX 2036/14-2 }

\maketitle

\section{Introduction}

Cremona transformations are birational automorphism of the projective space, they were firstly systematically studied by L. Cremona in the 19th century and remains 
a major classical topic in Algebraic Geometry. The interest in monomial Cremona transformations, otherwise has flourished more recently as one can see, for example in \cite{V,GP,SV1,SV2,SV3}. 
Following the philosophy introduced in \cite{SV1,SV2} and shared by \cite{SV3,CS} we study the so called ``birational combinatorics'' meaning, see {\it loc. cit.}, the theory of characteristic-free rational
maps of the projective space defined by monomials, along with natural criteria for such maps to be birational. The central point of view intend that the criteria must reflect the monomial data, as otherwise one falls back in the general theory of birational
maps in projective spaces. \\

We deal with two classes of monomial Cremona maps and we want to stress that our main tool is basic graph theory and the theory of clutters that are naturally 
associated to them. The determinantal principle of birationality 
proved in \cite{SV2} and stated here in Proposition \ref{prop:DPB} is the fundamental result linking the combinatorics and the algebraic set up. 
The first treat quadratic Cremona monomial maps in an arbitrary number of variables. 
we reobtain a result of \cite{SV2} that 
classify all the monomial quadratic Cremona maps on a combinatoric way, see \cite[Prop. 5.1]{SV2} and Theorem \ref{prop:degreetwo} here. The second class we 
study is the monomial square free cubic transformations with at most six variables, for this class we give a complete classification up to isomorphism, using 
the action of $S^n$, the permutation group on this set.  \\

We now describe the contents of the paper in more detail. In the first section we highlight the combinatoric set up, the log matrix associated to a finite set 
of monomials having the same degree and the associated clutters in the square free case. We recall the determinantal principle of birationality \ref{prop:DPB} and the 
duality principle \ref{prop:duality} both founded in the Simis-Villareal paper \cite{SV2}. We also present the so called counting Lemmas \ref{processo}, \ref{isomorfas}, \ref{dual}, \ref{cone} and \ref{sequencia de grau} 
which are the main tools to enumerate square free Cremona monomial maps up to projective transformations by understanding the natural action of $S^n$, the permutation group. \\

In the third section we present two extremal constructions of monomial Cremona transformations, Proposition \ref{prop:DLP} and Corollary \ref{cor:PRP} that allow us 
to reobtain the classification Theorem for quadratic monomial Cremona transformations in arbitrary number of variables, see \cite{SV2}, Theorem \ref{prop:degreetwo} in a very natural and combinatoric way and 
reobtain the classification of square free monomial Cremona transformations in $\P^3$ and $\P^4$, see \cite{SV2} and Corollaries \ref{cor:p3} and \ref{cor:p4}. \\

In the fourth section we prove the main result of the paper, Theorem \ref{thm:main}, counting the square free monomial Cremona transformations in $\P^5$ up to Projective transformations. By the duality principle, Proposition \ref{prop:DPB}, the 
hard part of the enumeration is the cubic square free monomial Cremona maps in six variables, described in Proposition \ref{prop:d3n6t1}, \ref{prop:d3n6t2} and \ref{prop:d3n6t3}. 

\section{Combinatorics}

\subsection{The combinatoric setup}

Let $\K$ be a field and $\K[x_1,\ldots,x_n]$ with $n \geq 2$ be the polynomial ring. For $v = (a_1,\ldots,a_n) \in \N^n$ we denote by 
$\underline{x}^{v}=x_1^{a_1}\ldots x_n^{a_n}$ the associated monomial and $d=|v|=a_1+\ldots+a_n$ its degree. The vector is called the log vector of the polynomial.

\begin{defin}
For each set of monomials $F=\{f_1,\ldots,f_n\}$ with $f_j \in \K[x_1,\ldots,x_n]$ 
we can associate its log vectors
$$v_j=(v_{1j}, \ldots,v_{nj})$$
where $\underline{x}^{v_j}=f_j$. The log matrix associated to $F$ is the matrix $A_F=(v_{ij})_{n \times n}$, whose columns are the (transpose of the) log vectors of $f_j$. 
If all the monomials have the same degree $d \geq 2$, which is our case of interest, then the log matrix is $d$-stochastic. 
\end{defin}

The monomial $f_j$ is called square free if for all $x_i$, $i=1,\ldots,n$, we have $x_i^2 \not |f_j$. The set $F$ is called square free if all its monomials are square free.\\

Let $F=\{f_1,\ldots,f_n\}$ be a set of monomials of same degree $d$ with $f_i \in \K[x_1,\ldots,x_n]$ for $i=1,\ldots,n$. $F$ defines a rational map: 
$$\varphi_F: \P^{n-1} \dashrightarrow \P^{n-1}$$
given by $\varphi_F(\underline{x})=(f_1(\underline{x}):\ldots,f_n(\underline{x}))$. \\

The following definition has an algebro-geometric flavor, including an algebraic notion of birationality. It is very useful in this context, for more details see \cite{SV1, SV2, CS}. 
An ordered set $F$ of $n$ monomials of same degree $d$ is a Cremona set if the map $\varphi_F$ is a Cremona transformation.

\begin{defin} \label{defin:cremonaset} Let $\K[\underline{x}]=\K[x_1,\ldots,x_n]$. 
Let $\K[\underline{x}_d]$ be the Veronese algebra generated by all monomials of degree $d$.
Let $F$ be a set of monomials of same degree. $F$ is a Cremona set if the extension $\K[f_1,\ldots,f_n] \subset \K[\underline{x}_d]$ 
becomes an equality on level of field of fractions.  
\end{defin}

We are interested in discuss whether the rational map $\Phi_F$ is a birational map. Hence we assume the following restrictions to the set $F$. 

\begin{defin} \label{defin:restricoescanonocas} We say that a set $F = \{f_1,\ldots,f_n\}$ of monomials $f_i \in \K[x_1,\ldots,x_n]$ of same degree satisfies the canonical restrictions if:
\begin{enumerate}
 \item For each $j=1,\ldots,n$ there is a $k$ such that $x_j|f_k$;
 \item For each $j=1,\ldots,n$ there is a $k$ such that $x_j \not |f_k$.
\end{enumerate}
\end{defin}

Isomorphism between monomial Cremona sets are given by permutation. 
\begin{defin}
Let $\sigma \in S_n$ be a permutation of $n$ letters. For each monomial 
$f = \underline{x}^{v}$, with $v \in \N^n$, we denote $f_{\sigma}=\underline{x}^{\sigma(v)}$. 
For a finite set of monomials $F$ and $\sigma \in S_n$ we denote $F_{\sigma}=\{f_{\sigma}|f \in F\}$.  
 Two Cremona sets $F,F' \subset \K[x_1,\ldots,x_n]$ are said to be isomorphic if there is $\sigma \in S_n$ such that $F'=F_{\sigma}$.
\end{defin}

\begin{rmk}\rm
Notice that we are tacitly supposing that the order of the elements in $F$ are irrelevant. {\it A priori} a Cremona map is given 
by an ordered set of polynomials, but, as a matter of fact, the property of to be a Cremona transformation is invariant under permutations. 
One can see this by the algebraic definition of Cremona set, Definition \ref{defin:cremonaset}. From now on we think an isomorphism between Cremona monomial maps as a relabel of the set of 
variables and a reorder of the forms.  
\end{rmk}

We make systematic use of the following Determinantal Principle of Birationality (DPB for short) due to Simis and Villarreal, see \cite[Prop. 1.2]{SV1}.

\begin{prop}\cite{SV1}\label{prop:DPB} {\bf(Determinantal Principle of Birationality (DPB))} Let $F$ be a finite set of monomials
of the same degree $d$ and let $A_F$ be its log matrix. Then $F$ is a Cremona set if and only if $\det A_F = \pm d$.
\end{prop}

We can associate to each set of square free monomials, $F \subset \K[x_1, \ldots, x_n]$, a combinatoric structure 
called clutter, also known as Sperner family, see \cite{HLT} for more details. 

\begin{defin}
A clutter $S$ is a pair $S=(V,E)$ consisting of a finite set, the vertex set $V$, and a set of 
subsets of $V$, the edge set $E$ in such a way that no one of edges is contained in another one.
We say that a clutter $S$ is a $d$-clutter if all the edges have same cardinality $|e|=d$. 
\end{defin}

Let $F=\{f_1,\ldots,f_n\}$ be a set of square free monomials of same degree $d$ with $f_i \in \K[x_1,\ldots,x_n]$ and let $A=(v_{ij})$ be its log matrix. We define the clutter $S_F=(V,E)$ 
in the following way, $V = \{x_1, . . . , x_n\}$ is the vertex set, and $E=\{e_1,\ldots,e_n\}$ where 
$e_i=\{x_jx_i, \ j \in \{1,\ldots,n\}|v_{ij}=1\}$. Notice that all the edges have the same cardinality, $|e|=d$, hence $S$ is a $d$-clutter. 
There is a bijective correspondence between $d$-clutters and sets of square free monomials of degree $d$. In the present work we deal only with $d$-clutters, but 
for short, we say only clutter instead of to say $d$-clutter. Furthermore, all the clutter considered have the same number of vertex and edges. 

\begin{ex}\rm
 A simple graph $G=(V,E)$ is a $2$-clutter. A set of square free monomials of degree two is represented by a simple graph. If the set $F$ of monomials of degree two 
 contains also some squares they can be represented as loops in the graph. Hence a set of monomials of degree two always can be represented as a graph.  
\end{ex}

\begin{defin}
 A subclutter $S'$ of a clutter $S=(V,E)$ is a clutter $S'=(V',E')$ where $V' \subset V$ and $E' \subset E$. A subclutter of a clutter is called a cone 
 if there is a vertex $v \in V$ such that $v \in e$ for all $e \in E$. A maximal cone of a clutter is a cone of maximal cardinality.
If $C=(V,E)$ is a cone with vertex $v$, the base of $C$ is the clutter $B_C=(V',E')$ with $V'=V\setminus v$ and $E'=\{e \setminus v, \forall e \in E\}$.
Deleting an vertex $v$ of a clutter $S=(V,E)$ we obtain a subclutter $S \setminus v=(V',E')$ where $V'=V \setminus v$ and $E' = \{e \in E|v \not \in e\}$. 
 \end{defin}

 We recall a combinatoric notion of duality. 
 
 \begin{defin}
 Let $F$ be a set of square free monomials of the same degree $d$ with log-matrix $A_F = (v_{ij})_{n \times n}$, its dual complement is the set 
 $F^{\vee}$ of monomials whose log-matrix is $A_{F^{\vee}} = (1 - v_{ij})_{n \times n}$. From the clutter point of view, if $S_F=(V,E)$, then 
 the dual complement clutter (dual clutter for short), is the clutter $S^{\vee}:=S_{F^{\vee}}=(V,E^{\vee})$ where $E^{\vee}=\{V\setminus e|e \in E\}$. 
 \end{defin}

 The following basic principle is very useful in the classification, for a proof see \cite[Proposition 5.4]{SV2}.

\begin{prop}\cite{SV2} \label{prop:duality}{\bf(Duality Principle)} Let $F$ be a set of square free monomials in $n$ variables, of the same
degree $d$, satisfying the canonical restrictions. Then $F$ is a Cremona set if and only if $F^{\vee}$ is a Cremona set.
 \end{prop}

\subsection{Counting Lemmas}

We deal with the classification and the enumeration of monomial Cremona transformations up to linear isomorphism. In the monomial case such an isomorphism is produced by a permutation 
of the variables and a permutation of the monomials. We study the action of the group 
$S_n$ on the set of all square free Cremona sets of monomials of fixed degree.\\

\begin{defin}
Let $G$ be a group and $X$ a non empty set. We denote by
$$\begin{array}{cccc}
\ast: & G \times X & \rightarrow & X \\
 & (g,x) & \mapsto & g*x \end{array}$$
an action of $G$ on $X$. The action induces a natural equivalence relation among the elements of $X$; $x\equiv y$ if and only if, $x=g\ast y$ for some $g \in G$. We denote the orbit of a element $x \in X$ by
$$\mathcal{O}_x=\{x' \in X| x'=g \ast x \mbox{ for some}\  g \in G\}.$$
The set of the orbits is the quotient $X/G$. The stabilizer of $x$ is $G_x=\{g \in G | g*x=x\}<G$. If $G$ acts on finite set $X$, then this action induces an action on $X^k$ the set of $k$-subsets of $X$.\\
 
\end{defin}

\begin{defin}
Let $\mathcal{M}_{n,d}$ be the set of square free monomials in $\K[x_1, \ldots, x_n]$ and let $\mathcal{M}^k_{n,d}$ be the set of $k$-subsets of $\mathcal{M}_{n,d}$. There is a natural action of $S_n$ on both the sets $\mathcal{M}_{n,d}$ and $\mathcal{M}^k_{n,d}$.

$$\begin{array}{cccc}
\ast: & S_n \times \mathcal{M}_{n,d}  & \rightarrow & \mathcal{M}_{n,d} \\ &(\sigma,m)=(\sigma, x_1^{a_1}\ldots x_n^{a_n}) & \mapsto & \sigma \ast m = x_{\sigma(1)}^{a_1}\ldots x_{\sigma(n)}^{a_n}
\end{array}$$
$$\begin{array}{cccc}
 \ast: & S_n \times \mathcal{M}^k_{n,d}  & \rightarrow & \mathcal{M}^k_{n,d}\\
& (\sigma,\{m_1, \ldots, m_k\}) & \mapsto & \{\sigma \ast m_1, \ldots, \sigma \ast m_k\}
\end{array}.$$

Let us consider the subset $\mathcal{C}_{n,d} \subset \mathcal{M}^n_{n,d}$ of Cremona sets representing square free monomial Cremona maps. 
To classify the square free  monomial Cremona maps up to linear isomorphism is equivalent to determine the orbits of $\mathcal{C}_{n,d}/S_n \subset \mathcal{M}^n_{n,d}/ S_n$.
 
\end{defin}

The next result allow us to construct $\mathcal{M}^k_{n,d}/S_n$ iteratively. Notice that $S_n$ acts transitively on $M^1_{n,d}$.

\begin{lema}\label{processo}
 If $\mathcal{M}^i_{n,d}/S_n=\{\mathcal{O}_{F_1}, \ldots, \mathcal{O}_{F_r}\}$ then $$\mathcal{M}^{i+1}_{n,d}/S_n=\{\mathcal{O}_{\{F_j,\beta*f\}}; \, F_j \mbox{ is a representative of an orbit on} \mathcal{M}^i_{n,d}/S_n,$$ $$\mbox{for some} \, f \in \mathcal{M}_{n,d} \setminus F_j\ \mbox{ and }\ \, \forall \beta \in S_n \,  \}$$
 \end{lema}

\begin{proof} Define $A=\{\mathcal{O}_{\{F_j,\beta*f\}}; \, F_j \mbox{ is a representative of some orbit in} \mathcal{M}^i_{n,d}/S_n,$ $\mbox{for some } \, f \in \mathcal{M}_{n,d} \setminus F_j \mbox{ and } \, \forall \beta \in S_n, \,  \}$. 
 Then $\mathcal{M}^{i+1}_{n,d}/S_n = A$.

In fact, $A \subset \mathcal{M}^{i+1}_{n,d}/S_n$. Consider $\mathcal{O}_G \in \mathcal{M}^{i+1}_{n,d}/S_n$, 
say that $G=\{g_1, \ldots, g_i, g_{i+1}\}$. Since $\overline{G}=G \setminus \{g_{i+1}\}$ has $i$ elements 
there is a $j$ such that $\overline{G} \in \mathcal{O}_{F_j}$, that is, there is $\gamma \in S_n$ such that 
$\gamma*F_j=\overline{G}$.

$S_n$ acts transitively in $\mathcal{M}_{n,d}$, hence, for $f \in \mathcal{M}_{n,d} \setminus F_j$, 
there is $\beta \in S_n$ such that $\beta*f=\gamma^{-1}*g_{i+1}$.

Therefore $\gamma*\{F_j, \beta*f\}=G$, that is, $\mathcal{O}_G=\mathcal{O}_{\{F_j, \beta*f\}} \in A$.
\end{proof}

This process determines $\mathcal{M}^i_{n,d}/S_n$, but each orbit can appear several times. We now 
answer partially the natural question of whether $\mathcal{O}_{\{F_j, \beta_1*f\}}=\mathcal{O}_{\{F_k, \beta_2*g\}}$, 
with $F_j, F_k$ representatives of distinct orbits in $\mathcal{M}^i_{n,d}/S_n$.

\begin{lema}\label{isomorfas}
Let $F \in \mathcal{M}^i_{n,d}$ and $f \in \mathcal{M}_{n,d}$. If $\gamma \in G_F$ and $\beta^{-1} \gamma \delta \in G_f$ then 
$\mathcal{O}_{\{F, \delta*f\}}=\mathcal{O}_{\{F, \beta*f\}}$.
In particular, if $\gamma \in G_F$ then $\mathcal{O}_{\{F, \gamma*f\}}=\mathcal{O}_{\{F, f\}}$.
\end{lema}
\begin{proof} Verify that $\gamma*\{F,\delta*f\}=\{F, \beta*f\}$.
\end{proof}

We leave the proof of the next trivial Lemma to the reader. 

\begin{lema}\label{dual}
If $F,G \in \mathcal{M}^i_{n,d}$ and $\mathcal{O}_{F}=\mathcal{O}_{G}$ then 
$\mathcal{O}_{\widehat{F}}=\mathcal{O}_{\widehat{G}}$.
\end{lema}

\begin{lema}\label{cone}
If $F,G \in \mathcal{M}^i_{n,d}$ and $\mathcal{O}_{F}=\mathcal{O}_{G}$, then for all maximal cone 
$\mathcal{C}$ of $F$ there is a just one maximal cone $\mathcal{C}'$ of $G$ such that 
$\mathcal{O}_{\mathcal{C}}=\mathcal{O}_{\mathcal{C}'}$ and {\it vice versa}.
\end{lema}

\begin{proof}
Since $\mathcal{O}_{F}=\mathcal{O}_{G}$ there is $\alpha \in S_n$ such that $\alpha * F=G$. 
Consider $\mathcal{C}$ a maximal cone of $F$, we will show that $\alpha * \mathcal{C} \subset G$ is a maximal cone of $G$. 
In fact, if $x_i$ is the vertex of the cone $\mathcal{C}$ then $x_i |f$ for all $f \in \mathcal{C}$ and 
$x_i$ does not divide any monomial of $F \setminus \mathcal{C}$, so $x_{\alpha(i)}|(\alpha * f)$ for all 
$f \in \mathcal{C}$ and $x_{\alpha(i)}$ does not divide any monomial of 
$G \setminus \alpha * \mathcal{C}$, that if $\alpha * \mathcal{C}$ is a cone of $G$. 
in order to prove the maximal condition for $\alpha * \mathcal{C}$, notice that 
$|\alpha * \mathcal{C}| = | \mathcal{C}|$.  In fact, if there exists a maximal cone 
$\mathcal{C'}$ of $G$ such that $| \mathcal{C'}| > |\alpha * \mathcal{C}|$ we have 
that $\alpha^{-1} * \mathcal{C'}$ is a cone of $F$ such that $|\alpha^{-1} * \mathcal{C'}|> | \mathcal{C}|$,
which contradicts the maximality of $\mathcal{C}$.
\end{proof}

Let $F \subset \K[x_1, \ldots, x_n]$ a set of square free monomials of degree $d$. 
the incidence degree of $x_i$ in $F$ is the number of monomials $f$ in $F$ such that $x_i|f$. 
The sequence of incidence degrees of $F$ is the sequence of incidence degree of $x_1, \ldots, x_n$ 
in non-increasing order. If the incidence degree of $x_i$ in $F$ is $a$ and $\alpha \in S_n$ such that $\alpha(i)=j$ 
then the incidence degree of $x_j$ in $\alpha * F$ is $a$. We have proved the following. 

\begin{lema}\label{sequencia de grau}
If $\mathcal{O}_{F}=\mathcal{O}_{G}$ then $F$ and $G$ have the same sequence of incidence degrees.
\end{lema}

\begin{rmk}\rm 
The converse is not true as one can check with easy examples. In fact, $F=\{x_1x_2,x_2x_3,x_1x_3,x_1x_4,x_4x_5,x_5x_6\}$ and 
$G=\{x_1x_2,x_2x_3,x_3x_4,x_4x_5,x_1x_5,x_1x_6\}$ have the same sequence of incidence degree, but $\mathcal{O}_{F} \neq \mathcal{O}_{G}$. 
It is easy to see that the graphs $G_F$ and $G_G$ associated to $F$ and $G$ respectively are non isomorphic. Indeed,  $G_F$ has just one cycle 
and it has three elements and on the other side $G_G$ has just one cycle and it has five elements, therefore $\mathcal{O}_{F} \neq \mathcal{O}_{G}$. 
\end{rmk}

\section{Two extremal principles and its consequences}

Consider a set $F=\{f_1,\ldots,f_n\}$ of monomials $f_i \in \K[x_1,\ldots,x_n]$, of same degree $d$ satisfying the canonical restrictions. The log matrix of $F$, $A_F$ is a $d$-stochastic matrix since 
the sum of every column is equal to $d$. The incidence degree of each vertex $x_i$ is denoted by $a_i$. By double counting in the log matrix, we have the following incidence equation:
\begin{equation}\label{eq:incidence}
a_1+\ldots+a_n=nd. 
\end{equation}

Where $1 \leq a_i \leq n-1$ for all $i=1,\ldots,n-1$.\\

Suppose that $F$ has a variable whose incidence degree is $1$. Let $S=(V,E)$ be the associated clutter. We want to interpret geometrically this extremal condition. 
The next definition was inspired in the similar notion on graphs, to be more precise, leafs of a tree. 

\begin{defin}
 Let $S=(V,E)$ be a clutter. A leaf in $S$ is a vertex $v \in V$ with incidence degree $1$. Let $e$ be the only edge containing $v$. 
 Then deleting the leaf $v$ we obtain a subclutter $S\setminus v = (V\setminus v,E \setminus e)$.
\end{defin}

The next result allow us to delete leaves of the clutters associated to square free Cremona sets to obtain other square free Cremona set, on a smaller 
ambient space or, {\it vice versa}, to attach leaves to square free Cremona sets. 

\begin{prop} \label{prop:DLP} {\bf (Deleting Leaves principle (DLP))}
 Let $F=\{f_1,\ldots,f_n\}$ be a set of monomials of same degree $d$, with $f_i \in \K[x_1,\ldots,x_n]$ for $i=1,\ldots,n$, satisfying the canonical restrictions. 
  Suppose that $x_n|f_n$, $x_n^2\nmid f_n$, and $x_n \nmid f_i$ for $i = 1,\ldots,n-1$. 
 Set $F' = \{f_1,\ldots,f_{n-1}\}$ with $f_i \in \K[x_1,\ldots,x_{n-1}]$ for $i=1,\ldots,n-1$. $F'$ is considered as a set of monomials of degree $d$ in $\K[x_1,\ldots,x_{n-1}]$.
 \begin{enumerate}
  \item If $F'$ does not satisfy the canonical restrictions, then $F$ is not a Cremona set.
  \item If $F'$ satisfies the canonical restrictions. Then $F$ is a Cremona set if and only if $F'$ is a Cremona set.
 \end{enumerate}
\end{prop}

\begin{proof} Let $A_F$ be the log matrix of $F$. Computing the determinant $\det A_F$ by Laplace's rule on the last row it is easy to see that:
$$\det A_F = \pm \det A_{F'}.$$
\begin{enumerate}
 \item If $F'$ does not satisfy the canonical restrictions, then $\det A_{F'}=0$. Indeed, if there is a $x_j$ for some $j=1,\ldots,n-1$ such that $x_j \not | f_i$ 
 for all $i=1,\ldots,n-1$, then the $j$-th row of $A_{F'}$ is null and $\det A_{F'}=0$. On the other side, if if there is a $x_j$ for some $j=1,\ldots,n-1$ 
 such that $x_j | f_i$ 
 for all $i=1,\ldots,n-1$, then the $j$-th row of $A_{F'}$ has all entries equal to $1$. Since $A_{F'}$ is $d$-stochastic in the columns, replace the first row
 for the sum of all the rows except the $j$-th give us a row with all entries $d-1$, hence $\det A_{F'}=0$.  
 \item If $F'$ satisfies the canonical restrictions, then by the DPB, Proposition \ref{prop:DPB}, $F$ is a Cremona set if and only if $\det A_F=d$, since 
 $\det A_F = \pm \det A_{F'}$, the result follows. 
\end{enumerate}

\end{proof}

Suppose now, in the opposite direction, that $F$ has a variable whose incidence degree is $n-1$. Geometrically it is the maximal possible cone 
on the clutter, since $F$ satisfies the canonical restrictions. We want to focus on the base of this cone.

\begin{defin}
 Let $S=(V,E)$ be a clutter with $|E|=n$ we say that $v \in V$ is a root if it has incidence degree $n-1$. 
 We define a new clutter $S/v=(\tilde{V},\tilde{E})$ pucking the root $v$, where $\tilde{V}=V \setminus v$ and $\tilde{E}=\{e \setminus v|e \in E, v \in e\}$. 
 It is easy to see that $S/v=(S^{\vee}\setminus v)^{\vee}$. 
 \end{defin}

\begin{cor} \label{cor:PRP} {\bf (Plucking Roots Principle (PRP))}
 Let $F=\{f_1,\ldots,f_n\}$ be a set of square free monomials of same degree $d$, with $f_i \in \K[x_1,\ldots,x_n]$ for $i=1,\ldots,n$, satisfying the canonical restrictions. 
  Suppose that $f_i=x_ng_i$, for $i = 1,\ldots,n-1$ and $x_n \nmid f_n$. 
 Set $\tilde{F} = \{g_1,\ldots,g_{n-1}\}$ with $g_i \in \K[x_1,\ldots,x_{n-1}]$ for $i=1,\ldots,n-1$. 
 $\tilde{F}$ is considered as a set of square free monomials of degree $d$ in $\K[x_1,\ldots,x_{n-1}]$.
 \begin{enumerate}
  \item If $\tilde{F}$ does not satisfy the canonical restrictions, then $F$ is not a Cremona set.
  \item If $\tilde{F}$ satisfies the canonical restrictions. Then $F$ is a Cremona set if and only if $F'$ is a Cremona set.
 \end{enumerate}

\end{cor}

\begin{proof}
 Let $S_F$ the clutter associated to $F$. Since $S_F/v=(S_F^{\vee}\setminus v)^{\vee}=(S_{F^\vee}\setminus v)^{\vee}$ is the clutter associated to $\tilde{F}$, 
 hence $(\tilde{F})^{\vee}=(F^{\vee})'$. By Duality Principle, Proposition \ref{prop:duality}, $F$ is a Cremona set if and only if $F^{\vee}$ is a Cremona set.
 Since $F^{\vee}$ has a leaf, by DLP, Proposition \ref{prop:DLP}, we can delete it to obtain $(F^{\vee})'$ which is a Cremona set if and only if $F$ is a Cremona set. Since 
 Since $(F^{\vee})'=(\tilde{F})^{\vee}$, the result follows by Duality Principle.
 \end{proof}

To classify the Cremona sets of degree two we use the following Lemma. A proof can be found in \cite[Lemma 4.1]{SV2}.

\begin{lema}\cite{SV2} \label{lema:cohesive}
Let $F = {f_1, ... , f_n} \subset \K[\underline{x}] = \K[x_1,\ldots,x_n]$ be forms of fixed degree $d \geq 2$.
Suppose one has a partition $\underline{x} = \underline{y} \cup \underline{z}$ of the variables such that $F = G \cup H$, where the forms in
the set $G$ (respectively, $H$) involve only the $\underline{y}$-variables (respectively, $\underline{z}$-variables). If neither $G$ nor
$H$ is empty then $F$ is not a Cremona set.
\end{lema}

\begin{defin}
 A set $F$ of monomials satisfying the canonical restrictions is said to be cohesive if the forms can not be disconnected as in the hypothesis of Lemma \ref{lema:cohesive}. 
\end{defin}

The next result is very easy but it concentrate all the fundamental information over the leafless case.

\begin{lema}\label{lema:cycle}
 Let $C_n=(V,E)$ be a $n$-cycle with $n\geq 3$, $V=\{1,\ldots,n\}$ and\\ $E=\{\{1,2\},\{2,3\},\ldots,\{n-1,n\},\{n,1\}\}$. Let $M_n := M_{n \times n}$
 be the incidence matrix of $C_n$. Then $\det M_n = 1-(-1)^n$. 
\end{lema}

\begin{proof} Computing the determinant by Laplace's rule on the first row:\\
 $
\left| \begin{array}{cccccc}
  1&0&0&\ldots&0&1\\
  1&1&0&\ldots&0&0\\
  0&1&1&\ldots&0&0\\
  0&0&1&\ldots&0&0\\
  \vdots & \vdots & \vdots &\ldots&\vdots &\vdots \\
  0&0&0&\ldots&1&0\\
  0&0&0&\ldots&1&1\\

 \end{array} \right|_{n} = \left| \begin{array}{ccccc}
  
  1&0&\ldots&0&0\\
  1&1&\ldots&0&0\\
  0&1&\ldots&0&0\\
   \vdots & \vdots &\ldots&\vdots &\vdots \\
  0&0&\ldots&1&0\\
  0&0&\ldots&1&1\\

 \end{array} \right|_{n-1} + (-1)^{n-1}\left| \begin{array}{ccccc}
 
  1&1&0&\ldots&0\\
  0&1&1&\ldots&0\\
  0&0&1&\ldots&0\\
  \vdots & \vdots & \vdots &\ldots&\vdots \\
  0&0&0&\ldots&1\\
  0&0&0&\ldots&1\\

 \end{array} \right|_{n-1}$

 Since both determinants on the right are triangular, the result follows.
 \end{proof}

 We now are in position to give a direct and purely combinatoric proof of the structure result of monomial Cremona sets in degree two, see \cite[Prop. 5.1]{SV2}. 
 
\begin{prop}\label{prop:degreetwo} Let $F \subset \K[x_1,\ldots,x_n]$ be a cohesive set of monomials of degree two 
satisfying the canonical restrictions. Let $A_F$ denote the log matrix and $G_F$ the graph. The following conditions are equivalent:
\begin{enumerate}
  \item $\det A_F \neq 0$; 
 \item $F$ is a Cremona set;
 \item Either \begin{enumerate}
               \item[(i)] $G_F$ has no loops and a unique cycle of odd degree; 
               \item[(ii)] $G_F$ is a tree with exactly one loop.
              \end{enumerate}

\end{enumerate}
\end{prop}
 
\begin{proof}

Since $F$ is a cohesive set, the associated graph is connected. Therefore, by DLP, Proposition \ref{prop:DLP}, we can delete all the leaves of $G_F$ to construct another cohesive set $F'$ on 
$m\leq n$ variables whose graph is connected and leafless. 
The incidence Equation \ref{eq:incidence} applied to $F'$ give us $a'_1=\ldots=a'_m=2$. Hence $G_{F'}$ is a disjoint union of cycles and loops. 
Since $G_{F'}$ is connected, there are only two possibilities. Either

\begin{enumerate}
 \item[(i)] $G_{F'}$ is a single loop; $F'=\{x^2\}$ and $\det A_F=2$;
 \item[(ii)] or $G_{F'}$ is a cycle. By Lemma \ref{lema:cycle}, odd cycles has determinant $2$ and even cycles have determinant $0$. 
\end{enumerate}
The result follows by attaching the petals on $F'$ to reach $F$. 
\end{proof}

\begin{rmk} \label{trivialcases} \rm
 If we restrict ourself to the square free case, then for each $j=1,\ldots,n$, the log vectors $v_j=(v_{1j},\ldots,v_{nj})$
have entries $v_{ij}\in \{0,1\}$ and $d=\displaystyle \sum_{i=1}^n v_{ij} \leq \displaystyle \sum_{i=1}^n 1 = n$. Of course $d \neq n$ since $F$ has non common factors. 
The case $d=n-1$ is trivial since it implies that for each $j=1,\ldots,n$ there exists only one $v_{ij}=0$. Up to a permutation 
$v_{ij}=1-\delta_{ij}$ is the anti Kronecker's delta. Hence, for $d=n-1$ there is only one square free monomial Cremona map, which is just the standard involution: 
$$\varphi:\P^{n-1} \dashrightarrow \P^{n-1}$$
given by $\varphi(x_1:\ldots:x_n)=(x_1^{-1}:\ldots:x_n^{-1})=(x_2x_3\ldots x_n:x_1x_3\ldots x_n:\ldots:x_1x_2\ldots x_{n-1})$.
\end{rmk}

For $d\leq n-2$ the classification is significantly more complicated, we get the classification for $n \leq 6$. The next two results are also contained in \cite{SV2}.

\begin{cor}\label{cor:p3}
There are three non isomorphic square free monomial Cremona transformations in $\P^3$. 
\end{cor}

\begin{proof}
 We are in the case $n=4$. Let $d$ be the common degree. If $d=1$ we have only the identity. If $d=3$ we have only the standard inversion, see Remark \ref{trivialcases}.

  If $d=2$, by Theorem \ref{prop:degreetwo} there are only one possible Cremona set whose graph is:

  \hspace{2.0in}
  \begin{minipage}[h]{0.15 \linewidth}
  
\begin{center}
\includegraphics[width=1.0in,height=1.2in]{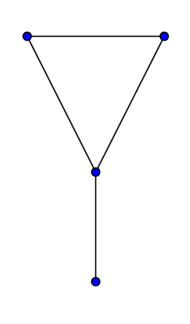}\\
\end{center}
\end{minipage}

\end{proof}

\begin{cor}\label{cor:p4}
There are ten non isomorphic square free monomial Cremona transformations in $\P^4$. 
\end{cor}

\begin{proof}

By Remark \ref{trivialcases} if the degree $d=1,4$ we have only one Cremona monomial map. \\

By Duality, \ref{prop:duality} the number of square free monomial Cremona maps of degree two and three are the same. Furthermore, 
by Theorem \ref{prop:degreetwo}, the possible square free Cremona sets of degree two have the following graphs:

\medskip

\hspace{0.4in}
\begin{minipage}[h]{0.15 \linewidth}
\begin{center}
\includegraphics[width=1.0in,height=1.2in]{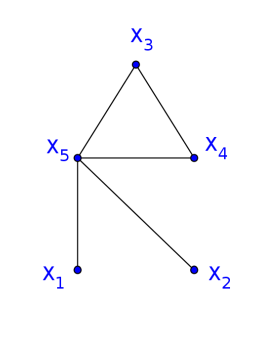}\\
\centering{Graph $G_1$}
\end{center}
\end{minipage}
\hspace{0.4in}
\begin{minipage}[h]{0.15 \linewidth}
\begin{center}
\includegraphics[width=1.0in,height=1.2in]{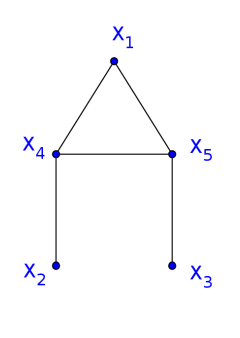}\\
\centering{Graph $G_2$}
\end{center}
\end{minipage}
\hspace{0.4in}
\begin{minipage}[h]{0.15 \linewidth}
\begin{center}
\includegraphics[width=1.0in,height=1.2in]{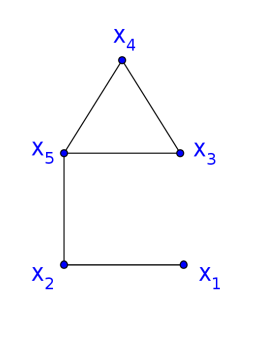}\\
\centering{Graph $G_3$}
\end{center}
\end{minipage}
\hspace{0.4in}
\begin{minipage}[h]{0.15 \linewidth}
\begin{center}
\includegraphics[width=1.0in,height=1.2in]{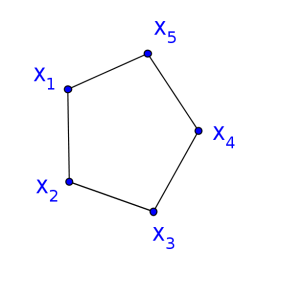}\\
\centering{Graph $G_4$}
\end{center}
\end{minipage}

\medskip

The dual complement of such Cremona sets are the square free Cremona sets of degree three, the associated clutter are the clutter duals of the preceding graphs:
 
\begin{minipage}[h]{0.15 \linewidth}
\begin{center}
\includegraphics[width=1.5in,height=1.5in]{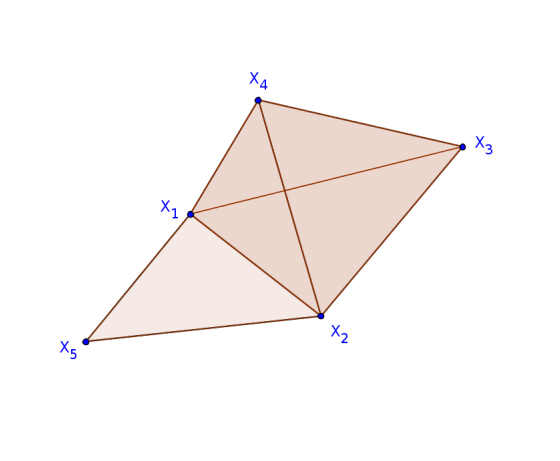}\\
\centering{Clutter $S_1 = G_1^{\vee}$}
\end{center}
\end{minipage}
\hspace{0.4in}
\begin{minipage}[h]{0.15 \linewidth}
\begin{center}
\includegraphics[width=1.4in,height=1.4in]{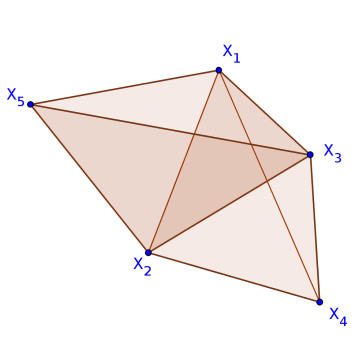}\\
\centering{Clutter $S_2 = G_2^{\vee}$}
\end{center}
\end{minipage}
\hspace{0.4in}
\begin{minipage}[h]{0.15 \linewidth}
\includegraphics[width=1.5in,height=1.5in]{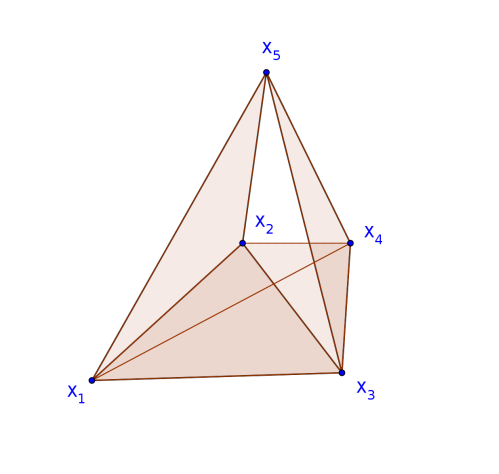}\\
\centering{Clutter $S_3 = G_3^{\vee}$}
\end{minipage}
\hspace{0.4in}
\begin{minipage}[h]{0.15 \linewidth}
\begin{center}
\includegraphics[width=1.5in,height=1.5in]{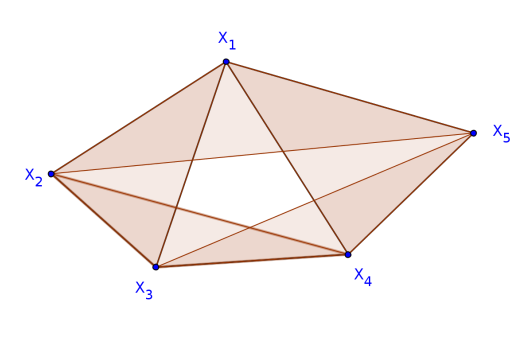}\\
\centering{Clutter $S_4 = G_4^{\vee}$}
\end{center}
\end{minipage}

\medskip

\end{proof}

\section{Square free monomial Cremona maps in $\mathbb{P}^5$}
 
Our main result is the following Theorem whose proof will be concluded in the next two sections. 

\begin{thm}\label{thm:main}
 There exist fifty-eight square free monomial Cremona transformations in $\P^5$ up to isomorphism.
\end{thm}

\begin{proof}
 Let $d$ be the degree of the Cremona set. Since $1 \leq d \leq 5$, we have only one Cremona set for each of the cases either $d=1$ or $d=5$, they are dual of each other. 
 For $d=2$, or dually $d=4$, we have eight possibilities, see Proposition \ref{prop:d2n6}. In total sixteen. \\
 For $d=3$ there are forty non isomorphic monomial Cremona sets according to Propositions \ref{prop:d3n6t1}, \ref{prop:d3n6t2} and \ref{prop:d3n6t3}.
\end{proof}

As a matter of fact we give a complete description of this Cremona sets drawing the associated clutters. 

\subsection{Square free monomial Cremona maps of degree two in $\mathbb{P}^5$} 
 
 \begin{prop}\label{prop:d2n6}
 
  There are, up to isomorphism, eight square free monomial Cremona sets of degree $2$ in $\K[x_1,\ldots,x_6]$. 
 \end{prop}

 \begin{proof} According to Theorem \ref{prop:degreetwo} we have the following possibilities for the associated graph of such a Cremona set. 

 \medskip
 
\hspace{0.4in}
\begin{minipage}[h]{0.15 \linewidth}
\hspace{0.4in}
\begin{center}
\includegraphics[width=1.0in,height=1.2in]{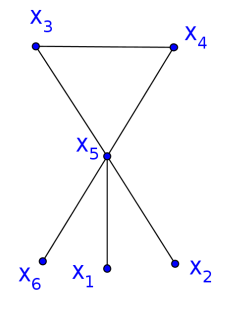}\\
\end{center}
\end{minipage}
\hspace{0.4in}
\begin{minipage}[h]{0.15 \linewidth}
\begin{center}
\includegraphics[width=1.0in,height=1.2in]{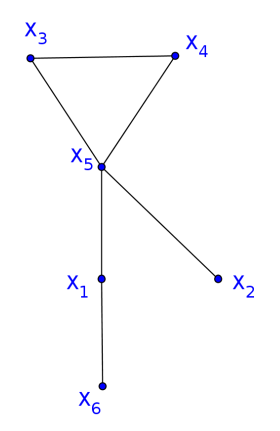}\\
\end{center}
\end{minipage}
\hspace{0.4in}
\begin{minipage}[h]{0.15 \linewidth}
\includegraphics[width=1.0in,height=1.2in]{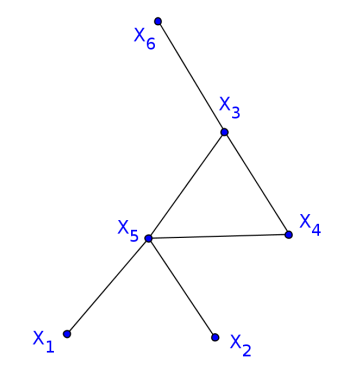}\\
\end{minipage}
\hspace{0.4in}
\begin{minipage}[h]{0.15 \linewidth}
\begin{center}
\includegraphics[width=1.0in,height=1.2in]{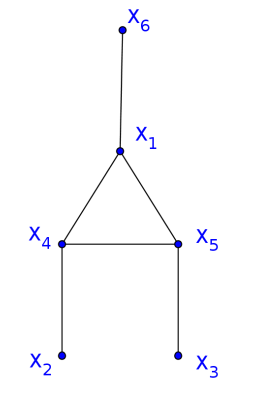}\\
\end{center}
\end{minipage}

\medskip

\hspace{0.4in}
\begin{minipage}[h]{0.15 \linewidth}
\hspace{0.4in}
\begin{center}
\includegraphics[width=1.0in,height=1.2in]{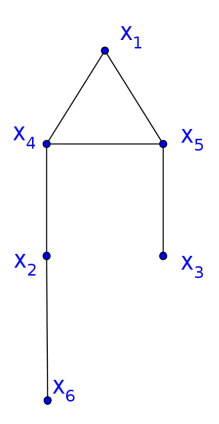}\\
\end{center}
\end{minipage}
\hspace{0.4in}
\begin{minipage}[h]{0.15 \linewidth}
\begin{center}
\includegraphics[width=1.0in,height=1.2in]{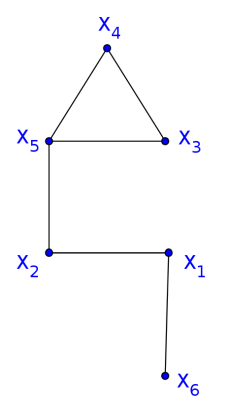}\\
\end{center}
\end{minipage}
\hspace{0.4in}
\begin{minipage}[h]{0.15 \linewidth}
\includegraphics[width=1.0in,height=1.2in]{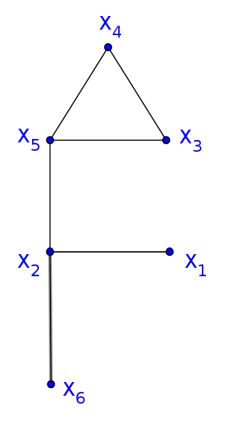}\\
\end{minipage}
\hspace{0.4in}
\begin{minipage}[h]{0.15 \linewidth}
\begin{center}
\includegraphics[width=1.0in,height=1.2in]{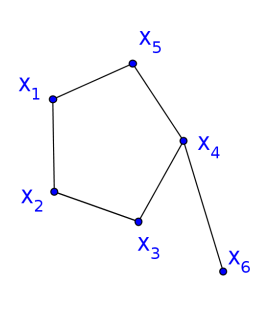}\\
\end{center}
\end{minipage}

 \end{proof}

\subsection{Square free monomial Cremona maps of degree three in $\mathbb{P}^5$}

Let us consider square free monomial Cremona transformations of $\mathbb{P}^5$
as a set of $n=6$ square free monomials of degree $d=3$. The corresponding log-matrix is a $6\times 6$ $3$-stochastic matrix 
whose determinant is $\pm 3$ by the DPB, Proposition \ref{prop:DPB}.

\begin{lema}\label{mdc}
Let $F=\{f_1,\ldots,f_6\} \subset \K[x_1,\ldots,x_6]$ be cubic monomials defining a Cremona transformation of $\mathbb{P}^5$.
Then for each choice of $4$ monomials of $F$ there are $2$ of them whose $\operatorname{gdc}$ is of degree $2$.
 \end{lema}
 \begin{proof} Suppose, by absurd, there are four monomials $f_1,f_2,f_3,f_4$ such that
 $\deg(\operatorname{gcd}(f_if_j)) \leq 1$. On the log matrix it imposes the existence of a $6 \times 4$ sub-matrix whose
 all $6 \times 2$ sub-matrices have at most one line with two entries $1$. Let us consider, up to isomorphism, $f_1=x_1x_2x_3$.
It is easy to see that, up to isomorphism, the log-matrix of these four vectors must be of the form:
$$\left[ \begin{array}{cccc}
1 & 1 & 0 & 0\\
1 & 0 & 1 & 0\\
1 & 0 & 0 & 1\\
0 & 1 & 1 & 0\\
0 & 1 & 0 & 1\\
0 & 0 & 1 & 1\\
\end{array} \right].$$
On the other side the log-matrix of any set
$F=\{x_1x_2x_3, x_1x_4x_5,x_2x_4x_6,x_3x_5x_6,f_5,f_6\}$
has even determinant. This contradicts our hypothesis that the set of monomials defines a Cremona transformation.
\end{proof}

\begin{rmk}\label{rmk:3types}\rm 
From now on we deal with the following setup: $A_F$ is a $6 \times 6$ $3$-stochastic matrix with eighteen entries $0$ and eighteen entries $1$ and
whose determinant is $\pm 3$. Furthermore, since the dimension $\dim R = 6$ and the degree $d=3$ are not coprime, $A_F$ can not
be doubly stochastic, see \cite[Proposition 5.6]{SV2}.
So there is a row of the matrix $A_F$ with at least four entries $1$ and it has tree possible types:
\begin{enumerate}
\item $A_F$ does not have a row with five entries $1$ but has a row with only one entry $1$. The clutter has one leaf but has no root;

\item $A_F$ has a row with five entries $1$. The associated clutter has a root;

\item $A_F$ does not have a row with five entries $1$ neither a row with only one entry $1$. The clutter is leafless and has no root.
\end{enumerate}

\end{rmk}

\begin{prop} \label{prop:d3n6t1}
There are , up to isomorphism, ten sets of square free Cremona monomials of degree $3$ in $\K[x_1, \ldots, x_6]$
whose log-matrix is of type $1$.
\end{prop}

\begin{proof} Let $F \subset \K[x_1, \ldots, x_6]$ be such a set, that is, $F=\{F',mx_6\}$, by DLP, Proposition \ref{prop:DLP}, $F'$ is a Cremona set.\\

Square free monomial Cremona transformations of degree $3$ in $\mathbb{P}^4$ were described in \cite{SV2}. 
There are $4$ of them, up to isomorphism. The associated clutters have the following representation:

\begin{minipage}[h]{0.15 \linewidth}
\begin{center}
\includegraphics[width=1.5in,height=1.5in]{tipo1clutter1.png}\\
\centering{Clutter $F'_1$}
\end{center}
\end{minipage}
\hspace{0.4in}
\begin{minipage}[h]{0.15 \linewidth}
\begin{center}
\includegraphics[width=1.5in,height=1.5in]{tipo1clutter2.png}\\
\centering{Clutter $F'_2$}
\end{center}
\end{minipage}
\hspace{0.4in}
\begin{minipage}[h]{0.15 \linewidth}
\includegraphics[width=1.5in,height=1.5in]{tipo1clutter3.png}\\
\centering{Clutter $F'_3$}
\end{minipage}
\hspace{0.4in}
\begin{minipage}[h]{0.15 \linewidth}
\begin{center}
\includegraphics[width=1.5in,height=1.5in]{tipo1clutter4.png}\\
\centering{Clutter $F'_4$}
\end{center}
\end{minipage}

\medskip

Therefore, $F$ is of the form $\{F'_i,m_ix_6\}$, where $i=1, 2, 3, 4$ and
$m \in \K[x_1, \ldots, x_5]$ is a square free monomial of degree $2$. Let us study the possibilities for the last monomial. According to Lemma \ref{isomorfas} there are some orbits that coincide.

First of all notice that the incidence degree of $x_1$ and $x_2$ in both, $F'_1$ and $F'_2$ is $4$, therefore
$F=\{F'_j,m_jx_6\}$, with $j=1,2$, satisfying our hypothesis imposes $m_j \in \K[x_3,x_4,x_5]$.

The stabilizer of $F'_1$ has generators $\beta=(1,2)$ and $\gamma=(3,4)$ and the stabilizer of $F'_2$ is generated by $\beta= (1,2)(4,5)$.
By Lemma \ref{isomorfas} we have $$\mathcal{O}_{\{F'_1,x_3x_5x_6\}}=\mathcal{O}_{\{F'_1,\gamma*(x_3x_5x_6)\}}. $$
By choosing one representative for each orbit we have two possibilities for $m_1$: $x_3x_4$ or $x_3x_5$, with associated clutters:

\medskip

\begin{center}
\begin{minipage}[h]{0.15 \linewidth}
\begin{center}
\includegraphics[width=1.5in,height=1.5in]{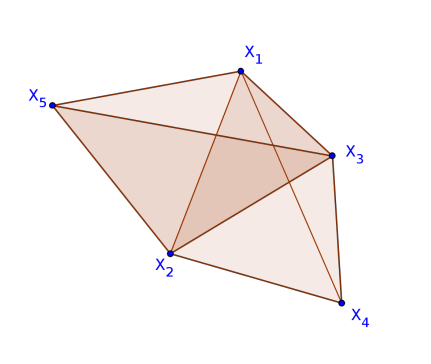}\\
\centering{Clutter $G_1$ $F'_1,x_3x_4x_6$}
\end{center}
\end{minipage}
\hspace{0.8in}
\begin{minipage}[h]{0.15 \linewidth}
\begin{center}
\includegraphics[width=1.5in,height=1.5in]{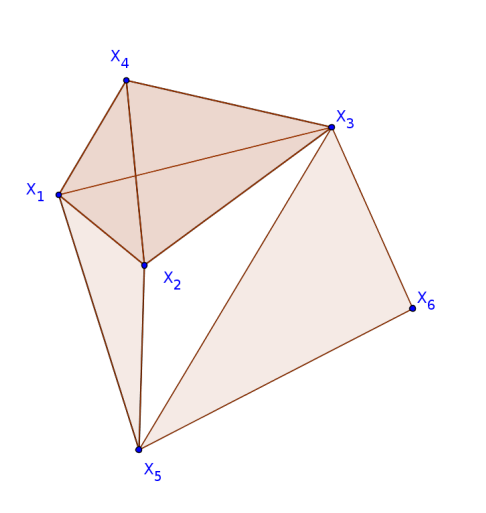}\\
\centering{Clutter $G_2$ $F'_1,x_3x_5x_6$}
\end{center}
\end{minipage}
\end{center}
\medskip

By Lemma \ref{isomorfas}: 
$$\mathcal{O}_{\{F'_2,x_3x_4x_6\}}=\mathcal{O}_{\{F'_2,\beta*(x_3x_4x_6)\}}. $$
We have also two possibilities for $m_2$: $x_3x_4$ or $x_4x_5$. The associated clutters are:

\medskip

\begin{center}
\begin{minipage}[h]{0.15 \linewidth}
\begin{center}
\includegraphics[width=1.5in,height=1.5in]{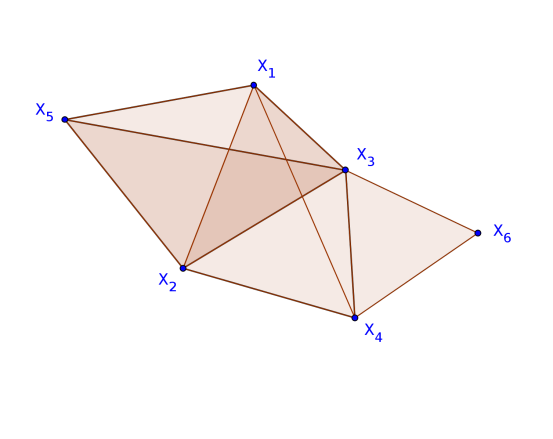}\\
\centering{Clutter $G_3$ $F'_2,x_3x_4x_6$}
\end{center}
\end{minipage}
\hspace{0.8in}
\begin{minipage}[h]{0.15 \linewidth}
\begin{center}
\includegraphics[width=1.5in,height=1.5in]{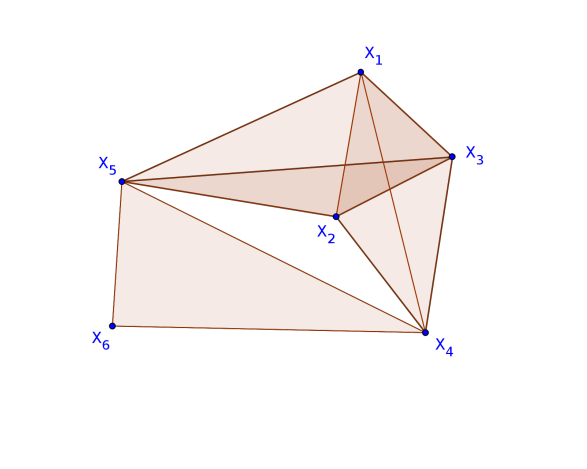}\\
\centering{Clutter $G_4$ $F'_2,x_4x_5x_6$}
\end{center}
\end{minipage}
\end{center}
\medskip

In the same way, the incidence degree of $x_1$ in $F'_3$ is $4$, so $m_3 \in \K[x_2,x_3,x_4,x_5]$.

The stabilizer of $F'_3$ is generated by $\beta= (3,4)$.  By the Lemma \ref{isomorfas} we have
$$\mathcal{O}_{\{F'_3,x_2x_3x_6\}}=\mathcal{O}_{\{F'_3,\beta*(x_2x_3x_6)\}}\ \mbox{and}\ \mathcal{O}_{\{F'_3,x_3x_5x_6\}}=\mathcal{O}_{\{F'_3,\beta*(x_3x_5x_6)\}}. $$
Taking one representative for each orbit, the last monomial can be: $x_2x_3x_6$, $x_2x_5x_6$, $x_3x_4x_6$ or $x_3x_5x_6$. The associated clutters are:

\begin{minipage}[h]{0.15 \linewidth}
\begin{center}
\includegraphics[width=1.5in,height=1.5in]{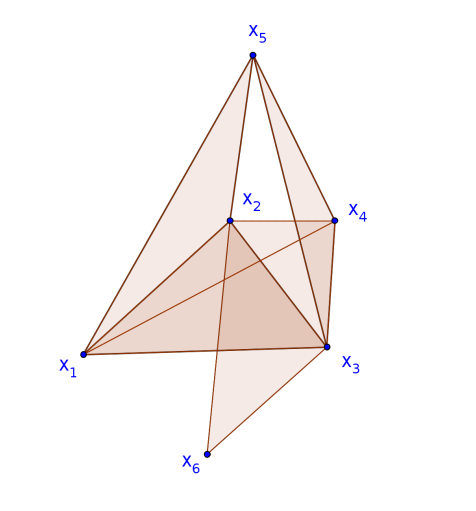}\\
\centering{Clutter $G_5$ $F'_3,x_2x_3x_6$}
\end{center}
\end{minipage}
\hspace{0.4in}
\begin{minipage}[h]{0.15 \linewidth}
\begin{center}
\includegraphics[width=1.5in,height=1.5in]{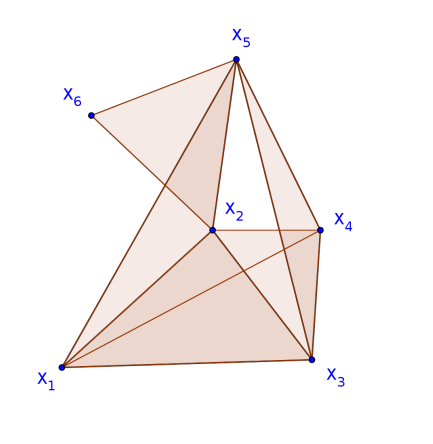}\\
\centering{Clutter $G_6$ $F'_3,x_2x_5x_6$}
\end{center}
\end{minipage}
\hspace{0.4in}
\begin{minipage}[h]{0.15 \linewidth}
\includegraphics[width=1.5in,height=1.5in]{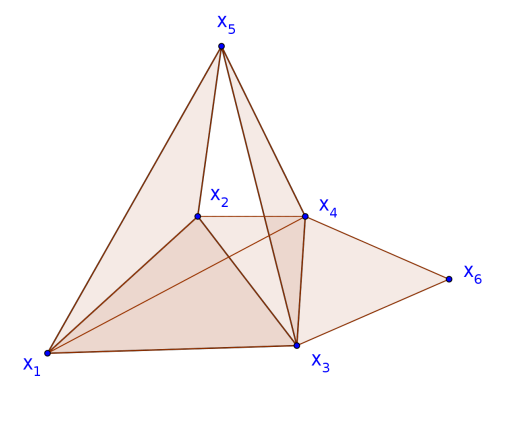}\\
\centering{Clutter $G_7$ $F'_3,x_3x_4x_6$}
\end{minipage}
\hspace{0.4in}
\begin{minipage}[h]{0.15 \linewidth}
\begin{center}
\includegraphics[width=1.5in,height=1.5in]{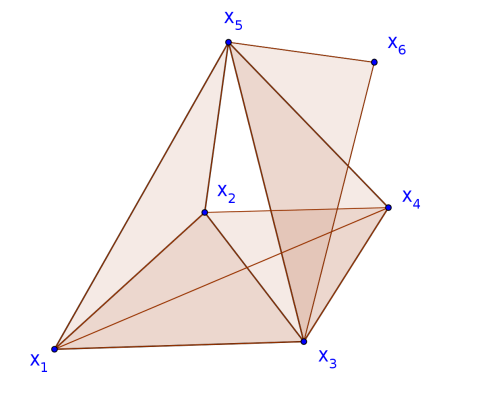}\\
\centering{Clutter $G_8$ $F'_3,x_3x_5x_6$}
\end{center}
\end{minipage}

\medskip

The stabilizer of $F'_4$ is generated by $\beta = (1,2,3,4,5)$.  By the Lemma \ref{isomorfas}
$$\mathcal{O}_{\{F'_4,x_1x_2x_6\}}=\mathcal{O}_{\{F'_4,\beta*(x_1x_2x_6)\}}= \mathcal{O}_{\{F'_4,\beta^2*(x_1x_2x_6)\}}= \mathcal{O}_{\{F'_4,\beta^3*(x_1x_2x_6)\}}= \mathcal{O}_{\{F'_4,\beta^4*(x_1x_2x_6)\}}, \mbox{and}$$ $$\mathcal{O}_{\{F'_4,x_1x_3x_6\}}=\mathcal{O}_{\{F'_4,\beta*(x_1x_3x_6)\}}=\mathcal{O}_{\{F'_4,\beta^2*(x_1x_3x_6)\}}= \mathcal{O}_{\{F'_4,\beta^3*(x_1x_3x_6)\}}=\mathcal{O}_{\{F'_4,\beta^4*(x_1x_3x_6)\}}. $$
The possibilities for the last monomial are $x_1x_2x_6$ or $x_1x_3x_6$. The associated clutters are:

\medskip

\begin{center}
\begin{minipage}[h]{0.15 \linewidth}
\begin{center}
\includegraphics[width=1.5in,height=1.5in]{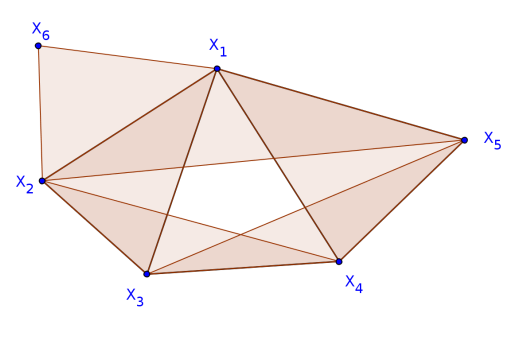}\\
\centering{Clutter $G_9$ $F'_4,x_1x_2x_6$}
\end{center}
\end{minipage}
\hspace{0.8in}
\begin{minipage}[h]{0.15 \linewidth}
\begin{center}
\includegraphics[width=1.5in,height=1.5in]{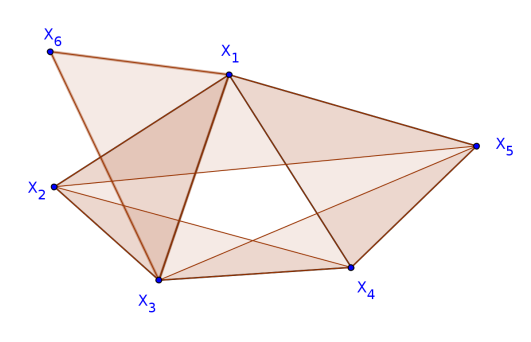}\\
\centering{Clutter $G_{10}$ $F'_4,x_1x_3x_6$}
\end{center}
\end{minipage}
\end{center}
\medskip

Using Lemma \ref{cone} and Lemma \ref{sequencia de grau} it is easy to see that these ten sets represent non isomorphic monomial Cremona transformations.\\

In fact, $G_1$ can not be isomorphic to other by the Lemma \ref{sequencia de grau}.
Notice also that for $i=2,\ldots,10$ the clusters $G_i$ have, up to isomorphism, four distinct types of maximal cones, whose bases can be represented by the following graphs.\\

\medskip

\begin{minipage}[h]{0.15 \linewidth}
\begin{center}
\includegraphics[width=1.0in,height=1.2in]{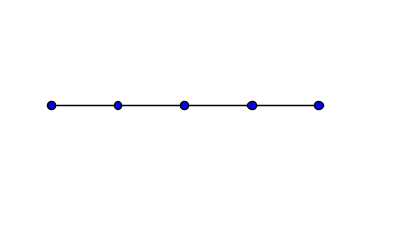}\\
\centering{Base $C_1$}
\end{center}
\end{minipage}
\hspace{0.4in}
\begin{minipage}[h]{0.15 \linewidth}
\begin{center}
\includegraphics[width=1.0in,height=1.2in]{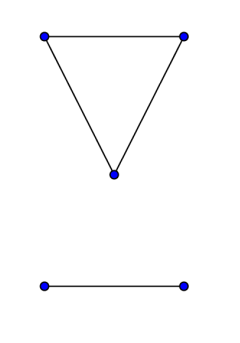}\\
\centering{Base $C_2$}
\end{center}
\end{minipage}
\hspace{0.4in}
\begin{minipage}[h]{0.15 \linewidth}
\begin{center}
\includegraphics[width=1.0in,height=1.2in]{fig5.png}\\
\centering{Base $C_3$}
\end{center}
\end{minipage}
\hspace{0.4in}
\begin{minipage}[h]{0.15 \linewidth}
\begin{center}
\includegraphics[width=1.0in,height=1.2in]{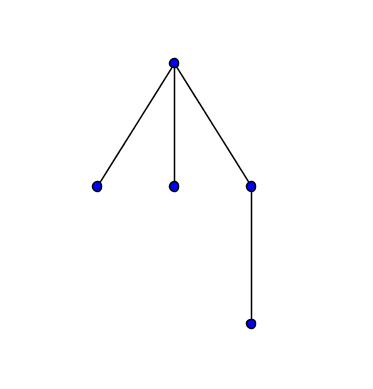}\\
\centering{Base $C_4$}
\end{center}
\end{minipage}

\medskip

The following matrix shows that the ten Cremona sets obtained are non isomorphic, by
having distinct incidence sequence or maximal cones.  

\medskip 

 \begin{center}
\begin{tabular}{|c|c|c|c|c|}
  \hline
$(4,4,4,3,2,1)$ & $G_2$ & $G_3$ & $G_5$ & $G_7$    \\
\hline

 MAXIMAL  &  $C_2$ e             &  $C_1$ e              &  $C_1, \, C_3$  & $C_3$ e  \\
CONE         & $C_3 (\times 2)$ & $C_3 (\times 2)$ & e $C_4$            & $C_4 (\times2)$ \\
   \hline
\end{tabular}

\end{center}

\medskip

 \begin{center}
\begin{tabular}{|c|c|c|c|c|c|}
  \hline
$(4,4,3,3,3,1)$ & $G_4$ & $G_6$ & $G_8$ & $G_9$ & $G_{10}$    \\
   \hline

 MAXIMAL  & $C_3 (\times 2)$  &  $C_3$ e $C_4$   &  $C_1$  e $C_3$  & $C_4 (\times 2)$ & $C_1 (\times 2)$  \\
CONE         &  &  &   &  & \\
\hline
\end{tabular}
\end{center}

\end{proof}

\begin{prop} \label{prop:d3n6t2}
There are , up to isomorphism, $20$ square free Cremona sets of degree $3$ in $\K[x_1, \ldots, x_6]$
of type $2$.
\end{prop}

\begin{proof} Let $F $ be such a set, that is, $F=\{x_1F',m\}$, where $F' \subset \K[x_2, \ldots, x_6]$ is a set of five square-free monomials of degree $2$ 
and $m \in \K[x_2, \ldots, x_6]$ is a degree $3$ square free monomial. By PRP, Corollary \ref{cor:PRP}, $F'$ is a Cremona set. 
Up to isomorphism there are four such Cremona set, see \cite{SV2} and also \ref{cor:p4}. The associated graphs can be represented as:

\medskip

\hspace{0.3in}
\begin{minipage}[h]{0.15 \linewidth}
\begin{center}
\includegraphics[width=1.0in,height=1.2in]{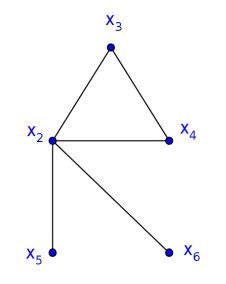}\\
\centering{Graph $F'_1$}
\end{center}
\end{minipage}
\hspace{0.4in}
\begin{minipage}[h]{0.15 \linewidth}
\begin{center}
\includegraphics[width=1.0in,height=1.2in]{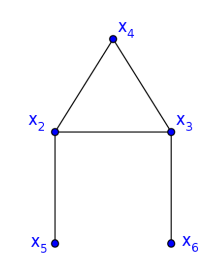}\\
\centering{Graph $F'_2$}
\end{center}
\end{minipage}
\hspace{0.4in}
\begin{minipage}[h]{0.15 \linewidth}
\begin{center}
\includegraphics[width=1.0in,height=1.2in]{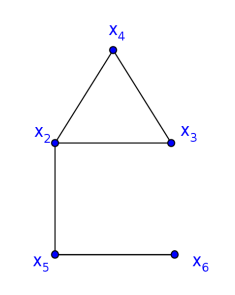}\\
\centering{Graph $F'_3$}
\end{center}
\end{minipage}
\hspace{0.4in}
\begin{minipage}[h]{0.15 \linewidth}
\begin{center}
\includegraphics[width=1.0in,height=1.2in]{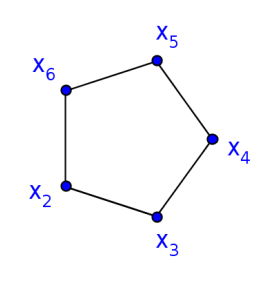}\\
\centering{Graph $F'_4$}
\end{center}
\end{minipage}

\medskip

So, $F$ is of the form $\{x_1F'_i,m_i\}$, where $i=1, 2, 3, 4$ and $m \in \K[x_2, \ldots, x_6]$
is a cubic monomial. Let us analyze the possibilities for $m_i$ according to the permutations that stabilize
$x_1F_i$, see Lemma \ref{isomorfas}, in order to exclude transformations in the same orbit.
The symmetries  of the graph associated to $F_i$ are useful.

The stabilizer of $F'_1$ is generated by $\beta=(3,4)$ and $\gamma=(5,6)$.
By the Lemma \ref{isomorfas} we have
$$\mathcal{O}_{\{x_1F'_1,x_2x_3x_5\}}=\mathcal{O}_{\{x_1F'_1,\beta*(x_2x_3x_5)\}}=\mathcal{O}_{\{x_1F'_1,\gamma*(x_2x_3x_5)\}}=\mathcal{O}_{\{x_1F'_1,\beta \gamma*(x_2x_3x_5)\}},$$
$$\mathcal{O}_{\{x_1F'_1,x_3x_4x_5\}}=\mathcal{O}_{\{x_1F'_1,\gamma*(x_3x_4x_5)\}}\ \mbox{ and} \ 
\mathcal{O}_{\{x_1F'_1,x_3x_5x_6\}}=\mathcal{O}_{\{x_1F'_1,\beta*(x_3x_5x_6)\}}. $$
So the possibilities for $m_1$ are $x_2x_3x_4$, $x_2x_3x_5$, $x_2x_5x_6$, $x_3x_4x_5$ and $x_3x_5x_6$. The associated clutters are the following ones:

\medskip

\hspace{0.6in}
\begin{minipage}[h]{0.15 \linewidth}
\begin{center}
\includegraphics[width=1.5in,height=1.5in]{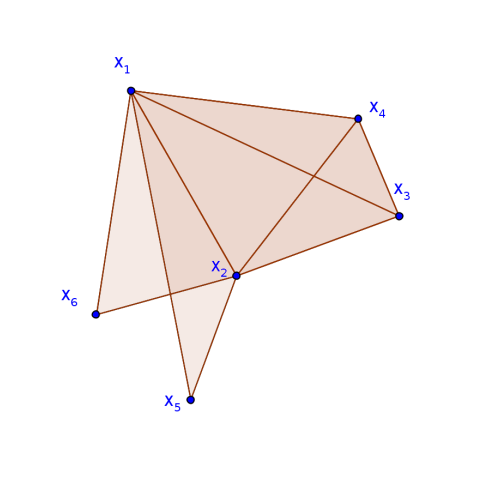}\\
\centering{Clutter $F_1$ $x_1F'_1,x_2x_3x_4$}
\end{center}
\end{minipage}
\hspace{0.6in}
\begin{minipage}[h]{0.15 \linewidth}
\begin{center}
\includegraphics[width=1.5in,height=1.5in]{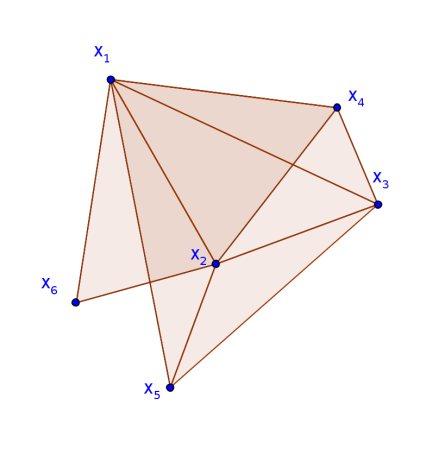}\\
\centering{Clutter $F_2$ $x_1F'_1,x_2x_3x_5$}
\end{center}
\end{minipage}
\hspace{0.6in}
\begin{minipage}[h]{0.15 \linewidth}
\begin{center}
\includegraphics[width=1.5in,height=1.5in]{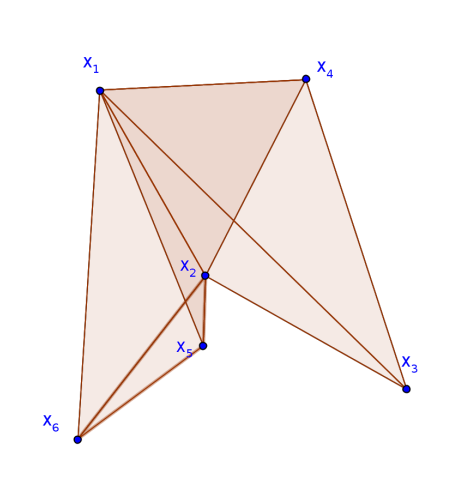}\\
\centering{Clutter $F_3$ $x_1F'_1,x_2x_5x_6$}
\end{center}
\end{minipage}

\medskip
\hspace{1.2in}
\begin{minipage}[h]{0.15 \linewidth}
\begin{center}
\includegraphics[width=1.5in,height=1.5in]{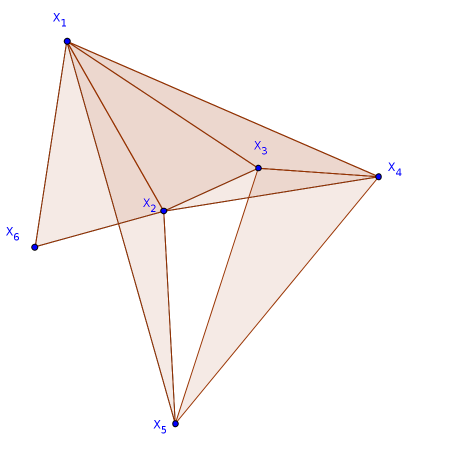}\\
\centering{Clutter $F_4$ $x_1F'_1,x_3x_4x_5$}
\end{center}
\end{minipage}
\hspace{0.8in}
\begin{minipage}[h]{0.15 \linewidth}
\begin{center}
\includegraphics[width=1.5in,height=1.5in]{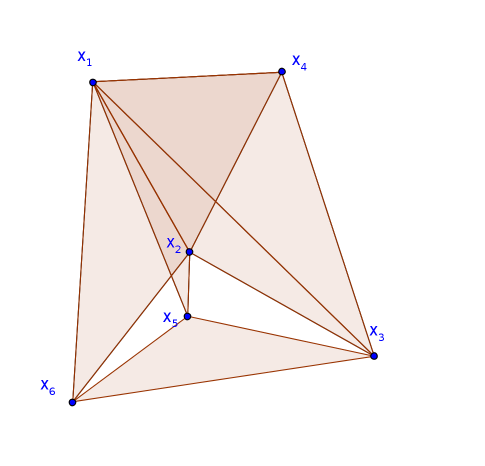}\\
\centering{Clutter $F_5$ $x_1F'_1,x_3x_5x_6$}
\end{center}
\end{minipage}

\medskip

The stabilizer of $F'_2$ is generated by $\beta=(2,3)(5,6)$. By the Lemma \ref{isomorfas} we have
$$\mathcal{O}_{\{x_1F'_2,x_2x_3x_5\}}=\mathcal{O}_{\{x_1F'_2,\beta*(x_2x_3x_5)\}}, \, \mathcal{O}_{\{x_1F'_2,x_2x_4x_5\}}=\mathcal{O}_{\{x_1F'_2,\gamma*(x_2x_4x_5)\}},$$
$$\mathcal{O}_{\{x_1F'_2,x_2x_4x_6\}}=\mathcal{O}_{\{x_1F'_2,\beta*(x_2x_4x_6)\}} \mbox{ and } \,
\mathcal{O}_{\{x_1F'_2,x_2x_5x_6\}}=\mathcal{O}_{\{x_1F'_2,\beta*(x_2x_5x_6)\}}. $$
The possibilities for $m_2$ are $x_2x_3x_4$, $x_2x_3x_5$, $x_2x_4x_5$, $x_2x_4x_6$, $x_2x_5x_6$ and $x_4x_5x_6$. The clutters associated to each of them are:

\medskip

\begin{center}
\begin{minipage}[h]{0.15 \linewidth}
\begin{center}
\includegraphics[width=1.5in,height=1.5in]{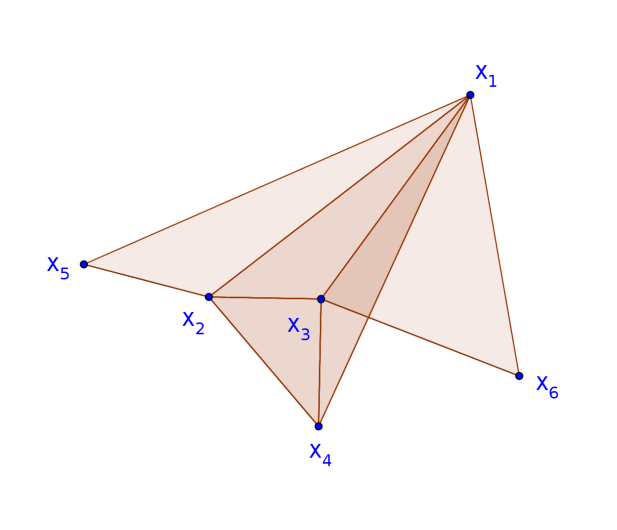}\\
\centering{Clutter $F_6$ $x_1F'_2,x_2x_3x_4$}
\end{center}
\end{minipage}
\hspace{0.8in}
\begin{minipage}[h]{0.15 \linewidth}
\begin{center}
\includegraphics[width=1.5in,height=1.5in]{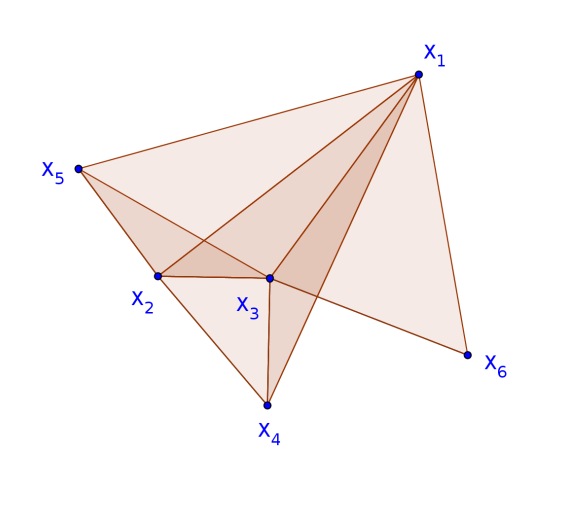}\\
\centering{Clutter $F_7$ $x_1F'_2,x_2x_3x_5$}
\end{center}
\end{minipage}
\hspace{0.8in}
\begin{minipage}[h]{0.15 \linewidth}
\begin{center}
\includegraphics[width=1.5in,height=1.5in]{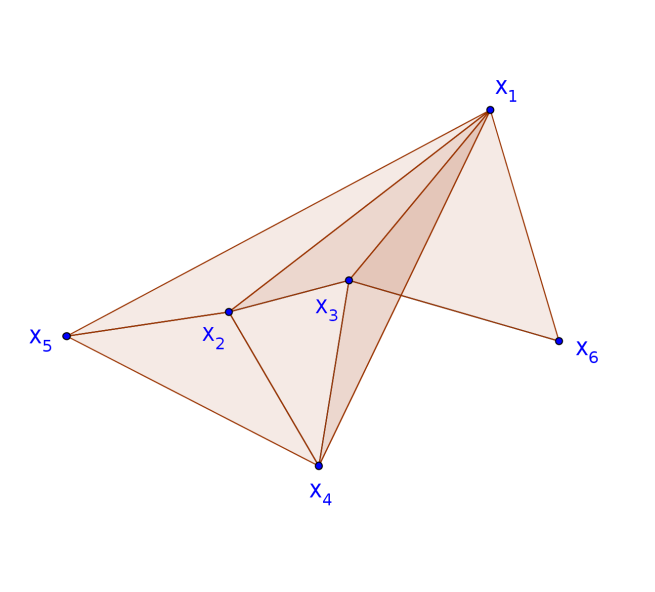}\\
\centering{Clutter $F_8$ $x_1F'_2,x_2x_4x_5$}
\end{center}
\end{minipage}
 
\end{center}

\medskip

\hspace{0.6in}
\begin{minipage}[h]{0.15 \linewidth}
\begin{center}
\includegraphics[width=1.5in,height=1.5in]{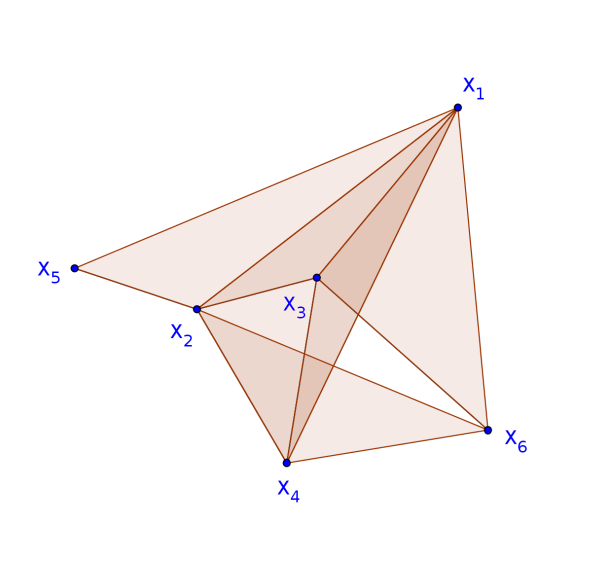}\\
\centering{Clutter $F_9$ $x_1F'_2,x_2x4x_6$}
\end{center}
\end{minipage}
\hspace{0.6in}
\begin{minipage}[h]{0.15 \linewidth}
\begin{center}
\includegraphics[width=1.5in,height=1.5in]{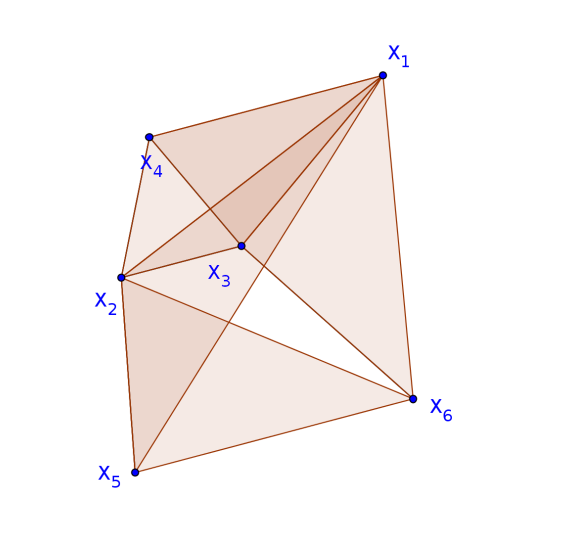}\\
\centering{Clutter $F_{10}$ $x_1F'_2,x_2x_5x_6$}
\end{center}
\end{minipage}
\hspace{0.6in}
\begin{minipage}[h]{0.15 \linewidth}
\begin{center}
\includegraphics[width=1.5in,height=1.5in]{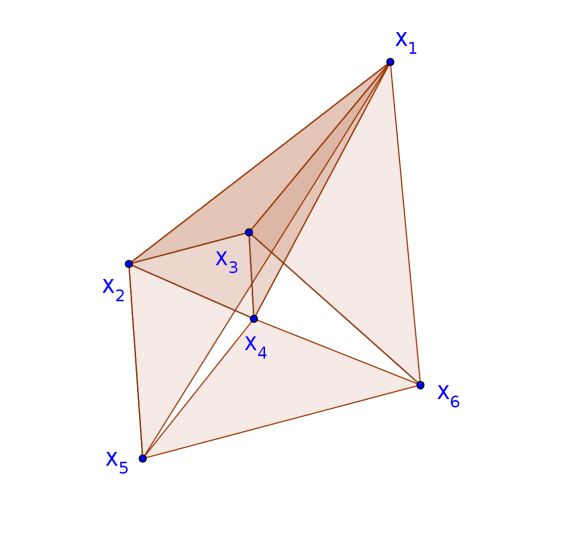}\\
\centering{Clutter $F_{11}$ $x_1F'_2,x_4x_5x_6$}
\end{center}
\end{minipage}

\medskip

The stabilizer of $F'_3$ is generated by $\beta=(3,4)$. By the Lemma \ref{isomorfas} we have
$$\mathcal{O}_{\{x_1F'_3,x_2x_3x_5\}}=\mathcal{O}_{\{x_1F'_3,\beta*(x_2x_3x_5)\}}, \, \mathcal{O}_{\{x_1F'_3,x_2x_3x_6\}}=\mathcal{O}_{\{x_1F'_3,\gamma*(x_2x_3x_6)\}} \mbox{and}$$
$$\mathcal{O}_{\{x_1F'_3,x_3x_5x_6\}}=\mathcal{O}_{\{x_1F'_3,\beta*(x_3x_5x_6)\}}. $$
The possibilities for $m_3$ are $x_2x_3x_4$, $x_2x_3x_5$, $x_2x_3x_6$, $x_2x_5x_6$, $x_3x_4x_5$, $x_3x_4x_6$ and $x_3x_5x_6$.

\medskip

\hspace{0.2in}
\begin{minipage}[h]{0.15 \linewidth}
\begin{center}
\includegraphics[width=1.3in,height=1.5in]{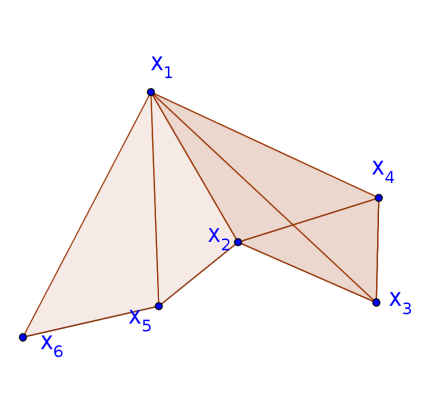}\\
\centering{Clutter $F_{12}$ $x_1F'_3,x_2x_3x_4$}
\end{center}
\end{minipage}
\hspace{0.4in}
\begin{minipage}[h]{0.15 \linewidth}
\begin{center}
\includegraphics[width=1.5in,height=1.5in]{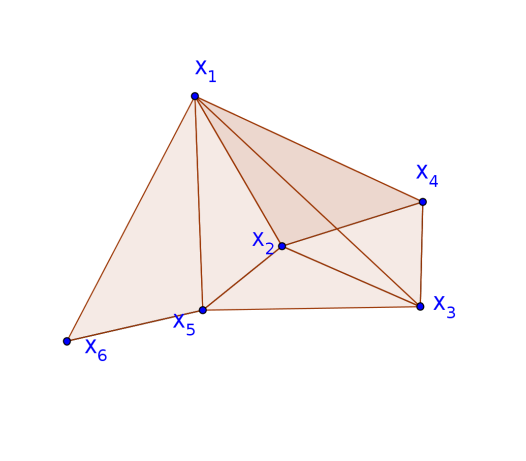}\\
\centering{Clutter $F_{13}$ $x_1F'_3,x_2x_3x_5$}
\end{center}
\end{minipage}
\hspace{0.4in}
\begin{minipage}[h]{0.15 \linewidth}
\begin{center}
\includegraphics[width=1.4in,height=1.4in]{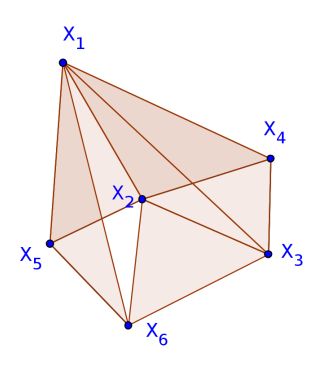}\\
\centering{Clutter $F_{14}$ $x_1F'_3,x_2x_3x_6$}
\end{center}
\end{minipage}
\hspace{0.4in}
\begin{minipage}[h]{0.15 \linewidth}
\begin{center}
\includegraphics[width=1.5in,height=1.5in]{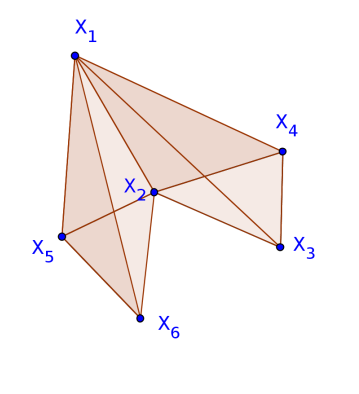}\\
\centering{Clutter $F_{15}$ $x_1F'_3,x_2x_5x_6$}
\end{center}
\end{minipage}

\medskip

\hspace{0.4in}
\begin{minipage}[h]{0.15 \linewidth}
\begin{center}
\includegraphics[width=1.5in,height=1.5in]{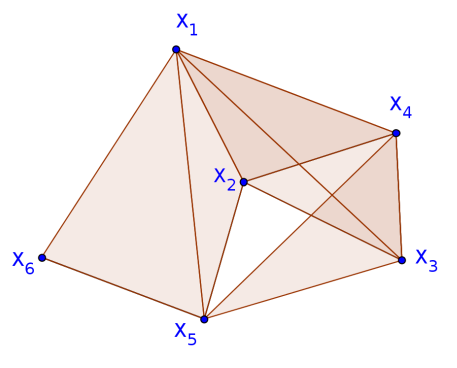}\\
\centering{Clutter $F_{16}$ $x_1F'_3,x_3x_4x_5$}
\end{center}
\end{minipage}
\hspace{0.6in}
\begin{minipage}[h]{0.15 \linewidth}
\begin{center}
\includegraphics[width=1.5in,height=1.5in]{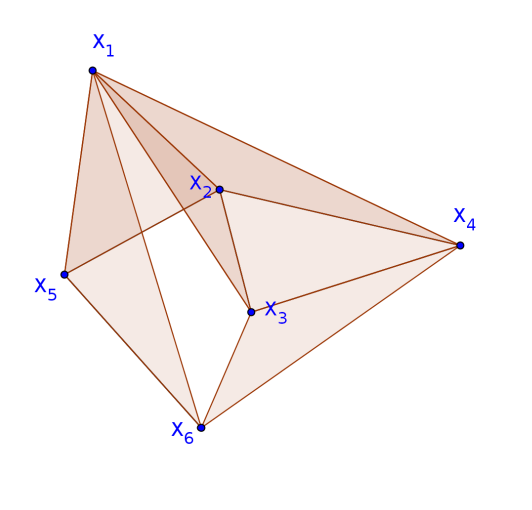}\\
\centering{Clutter $F_{17}$ $x_1F'_3,x_3x_4x_6$}
\end{center}
\end{minipage}
\hspace{0.6in}
\begin{minipage}[h]{0.15 \linewidth}
\begin{center}
\includegraphics[width=1.5in,height=1.5in]{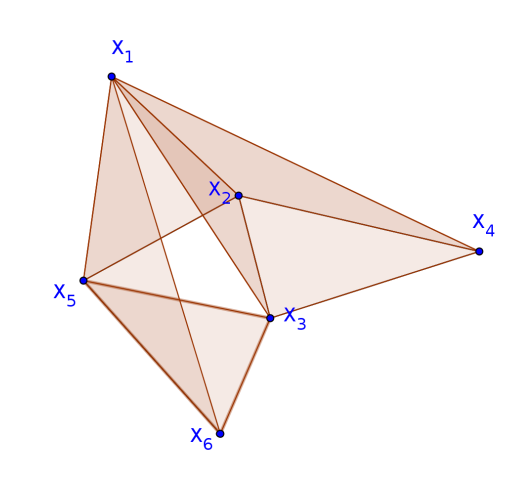}\\
\centering{Clutter $F_{18}$ $x_1F'_3,x_3x_5x_6$}
\end{center}
\end{minipage}

\medskip

The stabilizer for $F'_4$ is generated by $\beta=(2,3,4,5,6)$.
By the Lemma \ref{isomorfas} we have
\begin{center}
$\begin{array}{c}
\mathcal{O}_{\{x_1F'_4,x_2x_3x_4\}}=\mathcal{O}_{\{x_1F'_4,\beta*(x_2x_3x_4)\}}= \mathcal{O}_{\{x_1F'_4,\beta^2*(x_2x_3x_4)\}}= \\
=\mathcal{O}_{\{x_1F'_4,\beta^3*(x_2x_3x_4)\}}= \mathcal{O}_{\{x_1F'_4,\beta^4*(x_2x_3x_4)\}}, \mbox{ and }\\ \mathcal{O}_{\{x_1F'_4,x_2x_3x_5\}}=\mathcal{O}_{\{x_1F'_4,\beta*(x_2x_3x_5)\}}=\mathcal{O}_{\{x_1F'_4,\beta^2*(x_2x_3x_5)\}}=\\ =\mathcal{O}_{\{x_1F'_4,\beta^3*(x_2x_3x_5)\}}=\mathcal{O}_{\{x_1F'_4,\beta^4*(x_2x_3x_5)\}}.
\end{array}$
\end{center}
The possibilities for $m_4$ are $x_2x_3x_4$ and $x_2x_3x_5$. The associated clutters are:

\medskip

\begin{center}
\begin{minipage}[h]{0.15 \linewidth}
\begin{center}
\includegraphics[width=1.5in,height=1.5in]{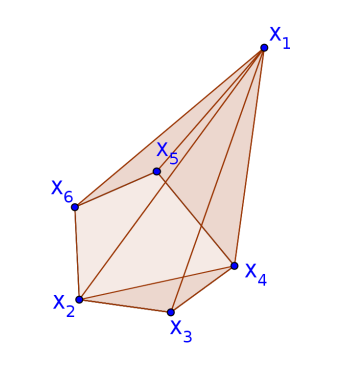}\\
Clutter $F_{19}$ $x_1F'_4,x_2x_3x_4$
\end{center}
\end{minipage}
\hspace{0.6in}
\begin{minipage}[h]{0.15 \linewidth}
\begin{center}
\includegraphics[width=1.5in,height=1.5in]{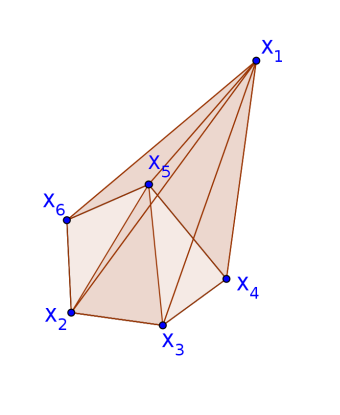}\\
Clutter $F_{20}$ $x_1F'_4,x_2x_3x_5$
\end{center}
\end{minipage}
\end{center}
\medskip

It is easy to check that these twenty sets define non isomorphic Cremona maps. Indeed, since $F_1, F_2, F_3, F_6, F_7, F_{16}$ are the only with its 
incidence sequence, they are non isomorphic among them and non isomorphic to any other by \ref{sequencia de grau}. 
Furthermore there are four (respectively five) non isomorphic square free Cremona sets of degree $ 3$
with incidence degree $(4,4,4,3,2,1)$ (respectively $(4,4,3,3,3,1)$) so, by the Lemma \ref{dual}, there are four
(respectively five) non isomorphic square free Cremona sets of degree $3$ with incidence degree $(5,4,3,2,2,2)$
(respectively $(5,3,3,3,2,2)$). Therefore $F_5, F_{10}, F_{11}, F_{14}, F_{15}, F_{17}, F_{18}, F_{19}$ and $F_{20}$
are non isomorphic among them and non isomorphic to the others.

The following matrix gives the maximal cones in the remaining cases.

\begin{center}
\begin{tabular}{|c|c|c|c|}
  \hline
  MAXIMAL CONES & \,\, $x_1F'_1$ \,\,  &  \,\, $x_1F'_2$  \,\,& \,\,  $x_1F'_3$ \,\,  \\
  \hline
   $(5,4,3,3,2,1)$  & $F_4$   &  $F_8$, $F_9$   &$F_{12}, \, F_{13}$  \\
\hline
 \end{tabular}
\end{center}

Now, the unique possible isomorphisms could be between $F_8, F_9$ and between $F_{12}, F_{13}$, but it is not the case. 
For instance its Newton dual have non isomorphic maximal cones. 

\end{proof}

\begin{prop} \label{prop:d3n6t3}
There are, up to isomorphism, $10$ square free monomial Cremona sets of degree $3$ in
$\K[x_1, \ldots, x_6]$ of type $3$.
\end{prop}

\begin{proof} After a possible reorder of the monomials and relabel of the variables, using Lemma \ref{mdc}, since $A_F$ does not have a row with five 
entries $1$, $F$ contains a subclutter having a maximal cone of the form:
$$C=\{x_1x_2x_3,x_1x_2x_4,x_1g_1, x_1g_2\}$$
Where $x_1 \not | g_1,g_2$ and the base of $C$ is a simple graph having $4$ edges and having at most $5$ vertices belonging to $\{x_2,\ldots,x_6\}$. 
There are six of such graphs up to isomorphism:

\medskip

\hspace{0.6in}
\begin{minipage}[h]{0.15 \linewidth}
\begin{center}
\includegraphics[width=1.0in,height=1.2in]{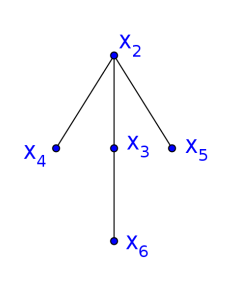}\\
Base of $C_1$
\end{center}
\end{minipage}
\hspace{0.6in}
\begin{minipage}[h]{0.15 \linewidth}
\begin{center}
\includegraphics[width=1.0in,height=1.2in]{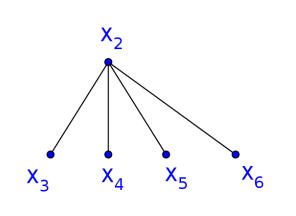}\\
Base of $C_2$
\end{center}
\end{minipage}
\hspace{0.6in}
\begin{minipage}[h]{0.15 \linewidth}
\begin{center}
\includegraphics[width=1.0in,height=1.2in]{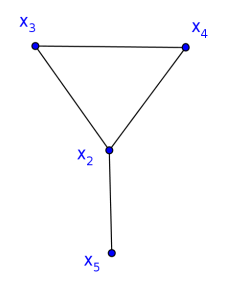}\\
Base of $C_3$
\end{center}
\end{minipage}

\medskip

\hspace{0.6in}
\begin{minipage}[h]{0.15 \linewidth}
\begin{center}
\includegraphics[width=1.0in,height=1.2in]{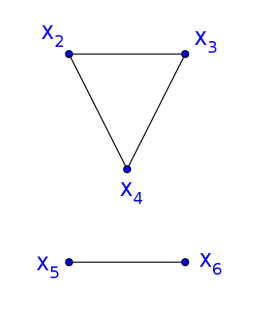}\\
Base of $C_4$
\end{center}
\end{minipage}
\hspace{0.6in}
\begin{minipage}[h]{0.15 \linewidth}
\begin{center}
\includegraphics[width=1.0in,height=1.2in]{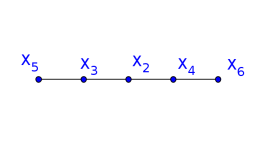}\\
Base of $C_5$
\end{center}
\end{minipage}
\hspace{0.6in}
\begin{minipage}[h]{0.15 \linewidth}
\begin{center}
\includegraphics[width=1.0in,height=1.2in]{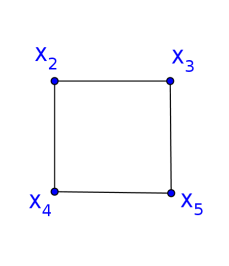}\\
Base of $C_6$
\end{center}
\end{minipage}

\medskip

\medskip

Setting $F_i=\{C_i, h_1,h_2\}$ for $i=1 \ldots 6$, where $h_1,h_2$ are cubic monomials such that $x_1 \not | h_1,h_2$. One can check that $\det(A_{F_i})=0$ for $i=2,3,6$. 
Therefore, if $F$ is a Cremona set in the hypothesis of the proposition, then either $F=F_1$ or $F=F_4$ or $F=F_5$.

Furthermore, imposing that $\det(A)_F = \pm 3$, and taking in account the permutations that stabilize the set, see Lemma \ref{isomorfas}, 
we have strong restrictions for $h_1,h_2$. 

The stabilizer of $C_1$ is generated by $\beta=(4,5)$. The possible values for the fifth monomial of $F$
 are: $x_2x_3x_6$, $x_2x_4x_6$ and $x_3x_4x_5$. There are five square free monomial Cremona sets of degree $3$ having $C_1$ as maximal cone. 
 The associated clutters are:
 
\medskip

\hspace{0.6in}
\begin{minipage}[h]{0.15 \linewidth}
\begin{center}
\includegraphics[width=1.5in,height=1.5in]{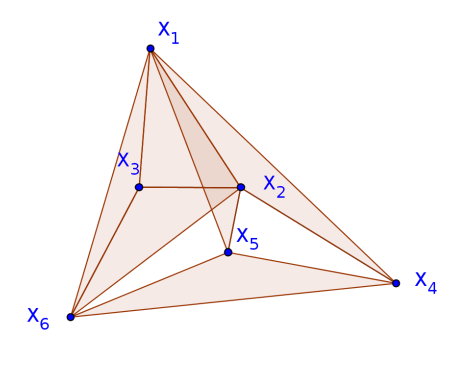}\\
\centering{Clutter $H_1$ $C_1,x_2x_3x_6, x_4x_5x_6$}
\end{center}
\end{minipage}
\hspace{0.6in}
\begin{minipage}[h]{0.15 \linewidth}
\begin{center}
\includegraphics[width=1.5in,height=1.5in]{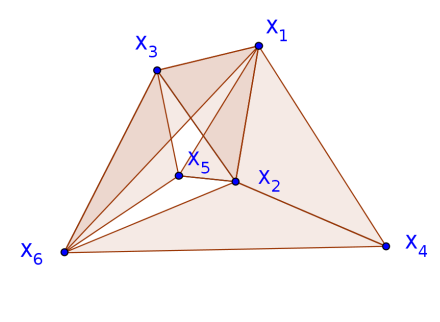}\\
\centering{Clutter $H_2$ $C_1,x_2x_4x_6, x_3x_5x_6$}
\end{center}
\end{minipage}
\hspace{0.6in}
\begin{minipage}[h]{0.15 \linewidth}
\begin{center}
\includegraphics[width=1.5in,height=1.5in]{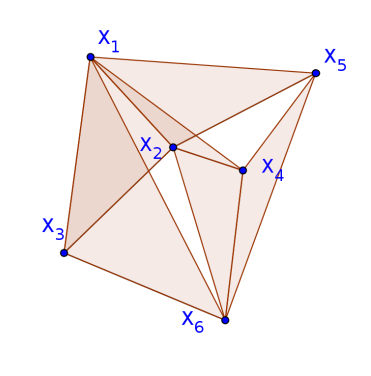}\\
\centering{Clutter $H_3$ $C_1,x_2x_4x_6, x_4x_5x_6$}
\end{center}
\end{minipage}

\medskip
\hspace{1.2in}
\begin{minipage}[h]{0.15 \linewidth}
\begin{center}
\includegraphics[width=1.5in,height=1.5in]{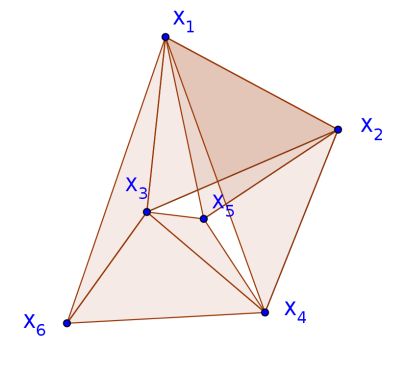}\\
\centering{Clutter $H_4$ $C_1,x_3x_4x_5, x_3x_4x_6$}
\end{center}
\end{minipage}
\hspace{1.2in}
\begin{minipage}[h]{0.15 \linewidth}
\begin{center}
\includegraphics[width=1.5in,height=1.5in]{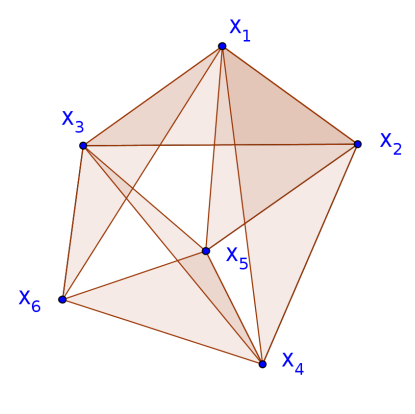}\\
\centering{Clutter $H_5$ $C_1,x_3x_4x_5, x_4x_5x_6$}
\end{center}
\end{minipage}

\medskip

The stabilizer of $C_4$ is generated by $\beta=(2,3)$, $\gamma=(3,4)$, $\delta=(2,4)$ and $\epsilon=(5,6)$. The fifth monomial of $F$ must be $x_2x_5x_6$.
We have two possible Cremona sets, whose clutters are:

\medskip
\hspace{1.2in}
\begin{minipage}[h]{0.15 \linewidth}
\begin{center}
\includegraphics[width=1.5in,height=1.5in]{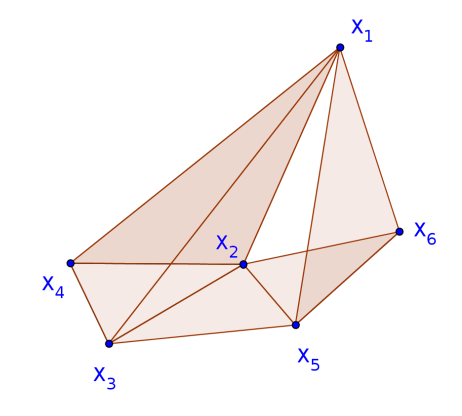}\\
\centering{Clutter $H_6$ $C_4,x_2x_3x_5, x_2x_5x_6$}
\end{center}
\end{minipage}
\hspace{1.2in}
\begin{minipage}[h]{0.15 \linewidth}
\begin{center}
\includegraphics[width=1.5in,height=1.5in]{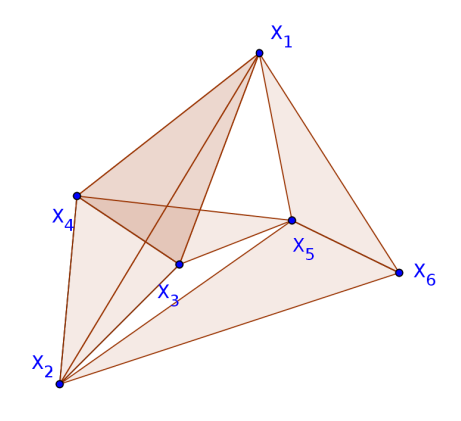}\\
\centering{Clutter $H_7$ $C_4,x_2x_5x_6, x_3x_4x_5$}
\end{center}
\end{minipage}

\medskip

The stabilizer of $C_5$ is generated by $\beta=(3,4)(5,6)$. The fifth monomial of $F$ is either: $x_2x_3x_4$, $x_2x_5x_6$ or $x_3x_4x_5$.
In this way we obtain six Cremona sets. The associated clutters are the following ones:

\medskip

\hspace{0.2in}
\begin{minipage}[h]{0.15 \linewidth}
\begin{center}
\includegraphics[width=1.3in,height=1.5in]{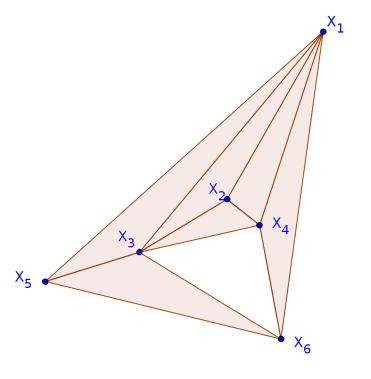}\\
\centering{Clutter $H_8$ $C_5,x_2x_3x_4, x_3x_5x_6$}
\end{center}
\end{minipage}
\hspace{0.4in}
\begin{minipage}[h]{0.15 \linewidth}
\begin{center}
\includegraphics[width=1.5in,height=1.5in]{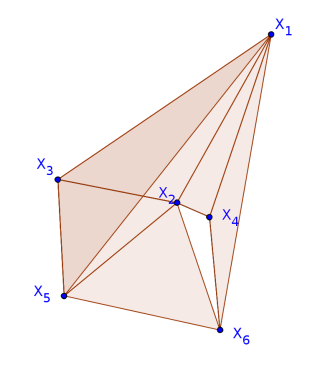}\\
\centering{Clutter $H_9$ $C_5,x_2x_5x_6, x_2x_3x_5$}
\end{center}
\end{minipage}
\hspace{0.4in}
\begin{minipage}[h]{0.15 \linewidth}
\begin{center}
\includegraphics[width=1.5in,height=1.5in]{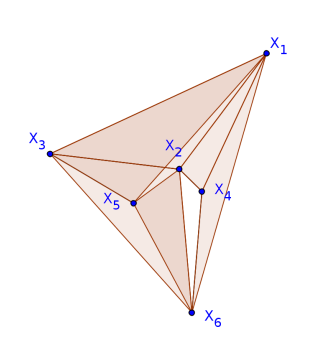}\\
\centering{Clutter $H_{10}$ $C_5,x_2x_5x_6, x_2x_3x_6$}
\end{center}
\end{minipage}
\hspace{0.4in}
\begin{minipage}[h]{0.15 \linewidth}
\begin{center}
\includegraphics[width=1.5in,height=1.5in]{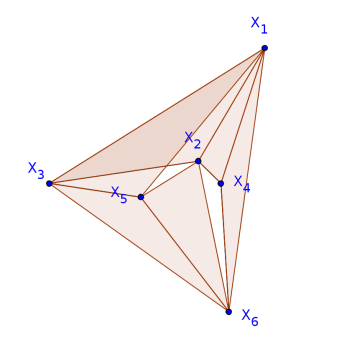}\\
\centering{Clutter $H_{11}$ $C_5,x_2x_5x_6, x_3x_5x_6$}
\end{center}
\end{minipage}

\medskip

\hspace{0.2in}
\begin{minipage}[h]{0.15 \linewidth}
\begin{center}
\includegraphics[width=1.3in,height=1.5in]{tipo3clutter5_1.png}\\
\centering{Clutter $H_{12}$ $C_5,x_3x_4x_5, x_2x_3x_6$}
\end{center}
\end{minipage}
\hspace{0.4in}
\begin{minipage}[h]{0.15 \linewidth}
\begin{center}
\includegraphics[width=1.5in,height=1.5in]{tipo3clutter5_2.png}\\
\centering{Clutter $H_{13}$ $C_5,x_3x_4x_5, x_2x_4x_6$}
\end{center}
\end{minipage}
\hspace{0.4in}
\begin{minipage}[h]{0.15 \linewidth}
\begin{center}
\includegraphics[width=1.5in,height=1.5in]{tipo3clutter5_3.png}\\
\centering{Clutter $H_{14}$ $C_5,x_3x_4x_5, x_3x_5x_6$}
\end{center}
\end{minipage}
\hspace{0.4in}
\begin{minipage}[h]{0.15 \linewidth}
\begin{center}
\includegraphics[width=1.5in,height=1.5in]{tipo3clutter5_4.png}\\
\centering{Clutter $H_{15}$ $C_5,x_3x_4x_5, x_4x_5x_6$}
\end{center}
\end{minipage}

\medskip

We have the following isomorphisms:
\begin{enumerate}
\item $(1,2)(3,4)*H_2=H_3$
\item $(1,3)(2,5,6)*H_4=H_{14}$
\item $(1,2,3)(4,6,5)*H_{10}=H_{12}$
\item $(1,2,3,5,6,4)*H_6=H_{13}$
\item $(1,3,2)(4,5)*H_8=H_9$
\end{enumerate}

Using the Lemma \ref{dual}, \ref{cone} and \ref{sequencia de grau}
one can easily check that $H_1$, $H_2$, $H_4$, $H_5$, $H_6$, $H_7$, $H_8$, $H_{10}$, $H_{11}$ and $H_{15}$ are non isomorphic. The result follows.

\end{proof}

{\bf Acknowledgments}

We would to thank Aron Simis for his sugestions on the theme and to Ivan Pan and Francesco Russo for their helpfull comments.\\
The third author was partially supported  by the CAPES postdoctoral fellowship, Proc. BEX 2036/14-2.

\end{document}